\documentclass[letterpaper,10pt]{IEEEtran}

\IEEEoverridecommandlockouts

\usepackage{color}
\usepackage{enumerate,graphicx,epstopdf}
\usepackage{amsmath,amssymb,bm,amsthm}
\usepackage{cite}
\usepackage{mathtools}
\usepackage{array}
\usepackage{algpseudocode,algorithm,algorithmicx}
\usepackage{color}
\usepackage{booktabs}
\usepackage[hyphens]{url}
\usepackage{lipsum}
\usepackage[dvipsnames]{xcolor}
\usepackage[deletedmarkup=sout,authormarkup=superscript]{changes}
\usepackage[normalem]{ulem}
\usepackage{flushend}

\usepackage[hidelinks]{hyperref}

\newtheorem{ass}{Assumption}
\newtheorem{rmk}{Remark}
\newtheorem*{prb}{Problem}
\newtheorem{dfn}{Definition}
\newtheorem{lmm}{Lemma}
\newtheorem{prp}{Proposition}
\newtheorem{cor}{Corollary}
\newtheorem{thm}{Theorem}

\newcommand{\ints}{\mathbb{Z}}
\newcommand{\R}{\mathbb{R}}

\newcommand{\reals}{\mathbb{R}}

\newcommand{\mc}[1]{\mathcal{#1}}
\newcommand{\bb}[1]{\mathbb{#1}}

\newcommand{\dom}[0]{\mathrm{Dom}~}

\newcommand{\esssup}[0]{\overline{\sup}}

\def\blue{\textcolor{black}}

\newcommand\veps[0]{\varepsilon}
\newcommand\eps[0]{\epsilon}

\newcommand\Kinf[0]{\mathcal{K}_\infty}
\newcommand\KL[0]{\mathcal{KL}}
\newcommand\K[0]{\mathcal{K}}
\renewcommand\L[0]{\mathcal{L}}
\newcommand\limsupk[0]{\underset{k\to\infty}{\overline{\lim}}}
\newcommand\sig[0]{\mathbb{L}}

\newcommand{\mto}[0]{\rightrightarrows}

\newcommand{\nc}{\mathcal{N}}
\newcommand{\proj}{\mathrm{proj}}
\newcommand\id[0]{\mathrm{id}}
\newcommand{\fix}{\mathrm{fix}}
\newcommand{\zer}{\mathrm{zer}}

\definechangesauthor[name={Giuseppe}, color=NavyBlue]{GB}
\definechangesauthor[name={Dominic}, color=RedViolet]{DLM}
\definechangesauthor[name={Mathias}, color=OliveGreen]{MHDB}

\begin{document}
\makeatletter
\newcommand{\printfnsymbol}[1]{%
  \textsuperscript{\@fnsymbol{#1}}%
}
\makeatother
\title{Online Feedback Equilibrium Seeking}

\author{Giuseppe Belgioioso\printfnsymbol{1}, Dominic Liao-McPherson\printfnsymbol{1}, Mathias Hudoba de Badyn, \\ Saverio Bolognani, Roy S.~Smith, John Lygeros, and Florian D\"{o}rfler
\thanks{\printfnsymbol{1}These authors contributed equally to this work.
G. Belgioioso, S. Bolognani, R. Smith, J. Lygeros, and F. D\"{o}rfler are with the ETH Z\"{u}rich Automatic Control Laboratory, \texttt{\{gbelgioioso, bsaverio, rsmith, jlygeros, dorfler\}@ethz.ch}. D. Liao-McPherson is with the University of British Columbia, \texttt{dliaomcp@mech.ubc.ca}. 
M.~Hudoba de Badyn is with the University of Oslo, \texttt{mathias.hudoba@its.uio.no}.
This work is supported by the SNSF via NCCR Automation (Grant 180545). 
}
}
\pagestyle{empty}
\maketitle
\thispagestyle{empty}
\begin{abstract}
This paper proposes a unifying design framework for dynamic feedback controllers that track solution trajectories of time-varying generalized equations, such as local minimizers of nonlinear programs or competitive equilibria (e.g., Nash) of non-cooperative games. Inspired by the feedback optimization paradigm, the core idea of the proposed approach is to re-purpose classic iterative algorithms for solving generalized equations (e.g., Josephy--Newton, forward-backward splitting) as dynamic feedback controllers by integrating online measurements of the continuous-time nonlinear plant. Sufficient conditions for closed-loop stability and robustness of the algorithm-plant cyber-physical interconnection are derived in a sampled-data setting by combining and tailoring results from (monotone) operator, fixed-point, and nonlinear systems theory. Numerical simulations on smart building automation and competitive supply-chain management are presented to support the theoretical findings.
\end{abstract}

\section{Introduction}
Online feedback optimization (FO) \cite{hauswirth2021optimization} is an emerging control paradigm for optimal steady-state operation of complex systems based on their direct closed-loop interconnection with optimization algorithms. FO controllers can handle control objectives beyond set-point regulation, typically tracking (a-priori unknown) solution trajectories of time-varying constrained optimization problems. In recent years, FO controllers have been proposed for a wide variety of problem settings
%
\cite{
hauswirth2021optimization, tang2017real, dall2016optimal,
simpson2020stability, chen2020distributed,
wang2011control,
bianchin2021time
}. These can be categorized by the type of control objective (e.g., convex or nonconvex) and constraints (e.g., hard or soft), the dynamics of the plant (e.g., nonlinear, linear, or algebraic), the type of algorithm (discrete or continuous-time), and the stability analysis (e.g., continuous-time, discrete-time, or hybrid), see \cite{hauswirth2021optimization} for a comprehensive list. FO has found widespread application in various domains, including power systems (e.g., for optimal power reserve dispatch \cite{tang2017real, dall2016optimal}, or frequency regulation in AC grids\cite{simpson2020stability, chen2020distributed}), communication networks (e.g., for network congestion control \cite{wang2011control}), and transportation systems (e.g., for ramp metering control \cite{bianchin2021time}).

These large-scale engineering infrastructures comprise multiple subsystems with local decision authority and preferences, commonly known as agents. These agents are typically self-interested, hence, a more general notion of ``efficiency'' is needed to model desirable (i.e., safe, locally optimal, and strategically stable) operating points for such systems, motivating the use of game-theoretic equilibria (e.g., Nash, Wardrop) \cite{belgioioso2022distributed}. A timely example are modern supply-chain systems, whereas suppliers, manufacturers, and retailers compete over the available resources (e.g., raw materials, market demand) to maximize their local profits \cite{spiegler2012control, anderson2010price}.

There are several classes of control approaches for driving non-cooperative multi-agent systems to steady-state configurations given by game-theoretic equilibria. One class uses first-order algorithms and builds on operator-theoretic and passivity arguments \cite{gadjov2018passivity,romano2019dynamic, bianchi2021continuous}. Passivity is leveraged in \cite{gadjov2018passivity} to design a distributed feedback control law to drive agents with single-integrator dynamics to a steady-state operating point given by a Nash equilibrium (NE). This approach is extended to multi-integrator agents affected by partially-known disturbances in \cite{romano2019dynamic}. The case of games with coupling constraints and agents with mixed-order integrator dynamics is considered for the first time in \cite{bianchin2021time}. All the these works consider agents coupled via their objective functions (and constraints) but with decoupled dynamics and use continuous-time flows. 

A second class of methods are sampled-data approaches based on the model-free \textit{extremum seeking} (ES) framework, for finding optima \cite{hazeleger2022sampled}, game-theoretic equilibria \cite{stankovic2011distributed, krilavsevic2021learning}, and solutions to dynamic inclusions (which subsume the two previous cases) \cite{poveda2017framework, poveda2017robust}. In \cite{stankovic2011distributed}, an ES controller is designed for NE seeking in games in which the agents have nonlinear decoupled dynamics. The extension to games with coupling constraints is presented in \cite{krilavsevic2021learning}. In both cases, practical stability is proven in the disturbance-free case. In \cite{poveda2017framework}, a general framework is developed for the design and analysis of a class ES controllers applicable to optimization as well as non-cooperative games. Practical stability is established by relying on the mathematical framework of hybrid dynamic inclusions. To accelerate convergence, \cite{poveda2017robust} extends the previous framework from periodic to event-triggered sampled-data control. These methods require nominal stability of the plant and, typically, time-scale separation assumptions. 

A third class of methods are extensions of economic model predictive control 
to non-cooperative systems, including multi-objective MPC \cite{zavala2012stability} and various flavours of game-theoretic MPC \cite{hall2022receding,van2002moving}. Unlike the first-order or ES methods, these consider transient operation and can handle unstable systems. In exchange, they require dynamic models of the plant and are typically computationally heavier due to the need to find accurate equilibria of trajectory games at each sampling time.

In this paper, we propose feedback equilibrium seeking (FES), an extension of FO that seeks to drive \blue{pre-stabilized} dynamical systems to ``efficient" operating points encoded by time-varying generalized equations (GEs). GEs contain constrained optimization as a special case and can model a broad range of equilibrium problems (e.g., Nash, Wardrop). FES controllers are most similar to the first class of control approaches discussed above \cite{gadjov2018passivity,romano2019dynamic, bianchi2021continuous}. 
They are typically based on first-order methods and require an (approximate) steady-state input-output sensitivity of the plant, \blue{which is assumed to stable}. As a result, they are faster than model-free ES methods; on the other hand, they are slower but computationally lighter than MPC-based methods, and require only static models.
\blue{Therefore, FES controllers occupy a unique middle ground between ES and MPC. Their greatest advantage emerges in situations where steady-state input-output sensitivity information is available, but detailed dynamic models are lacking. Such conditions are commonly found in sectors like power systems and process optimization \cite{ tang2017real, dall2016optimal, simpson2020stability, chen2020distributed}}.
Unlike \cite{gadjov2018passivity,romano2019dynamic, bianchi2021continuous}, our framework encompasses a broad class of algorithms and systems beyond continuous-time flows and multi-integrators. Moreover, we use discrete-time algorithms which is especially relevant for embedded distributed implementations, as communication in multi-agent \blue{engineering} networks is a discrete process. 

Some other recent works \cite{lu2020online, benenati2022optimal} study the problem of tracking time-varying generalized NEs (GNEs) with discrete-time algorithms.
A distributed algorithm for tracking the solution trajectory of an exogenously varying strongly monotone game is developed in \cite{lu2020online} and extended to monotone games in \cite{benenati2022optimal}. In \cite{agarwal2022game}, a GNE-seeking algorithm is adapted for online operation by integrating measurements from a physical system, however, the plant is treated as an algebraic map. All the above works neglect dynamic interactions between the plant and the algorithm (either the variation is exogenous or the plant is algebraic) while we consider closed-loop stability in the sampled-data case. This is significant as digital control of continuous-time plants is currently the dominant paradigm.

Our contribution to this area of research is threefold:
\begin{enumerate}[(i)]
\item We propose a general framework for designing FES controllers for continuous-time \blue{pre-stabilized} nonlinear systems by tapping into a broad class of first- and second-order discrete-time algorithms for generalized equations. This class is broad, and include many optimization and game-theoretic algorithms (e.g., sequential convex programming (SCP), proximal-gradient) as special cases.

\item \blue{We derive sufficient conditions for stability and robustness of the sampled-data algorithm-plant interconnection. Specifically, we prove practical input-to-state stability (Def.~\ref{def:LISS}) of the closed loop with respect to exogenous disturbances, under three fundamental conditions: 1. robust stability of the plant, 2. strong regularity of the control objective, and 3. robust convergence of the iterative algorithm. Moreover, we prove a sharper notion of input-to-state stability if an additional small-gain condition holds.}
%
As by-product of our analysis, we also demonstrate that some popular classes of algorithms are not suitable for constructing sampled-data FES (or FO) controllers.

\item We showcase the utility of our framework by designing two novel controllers based on a SCP algorithm for nonconvex nonsmooth optimization, and a forward-backward splitting controller for distributed Nash equilibrium seeking in dynamically-coupled games. Further, we illustrate their effectiveness through numerical simulations on smart building and supply chain examples.
\end{enumerate}

We build on our preliminary work \cite{9683614} which addresses the special case of strongly monotone GEs, globally linearly convergent algorithms, and exponentially stable plants. Our results here provide a significant extension that allows non-monotone GEs, locally-stable nonlinear plants and non-linearly locally convergent algorithms. This algorithmic extension is particularly important since it captures many practically relevant cases, such as algorithms for nonconvex optimization, which are not globally convergent, and primal-dual and distributed algorithms which typically exhibit sub-linear convergence only.



\begin{figure}[t]
    \centering
    \includegraphics[width=.9\columnwidth]{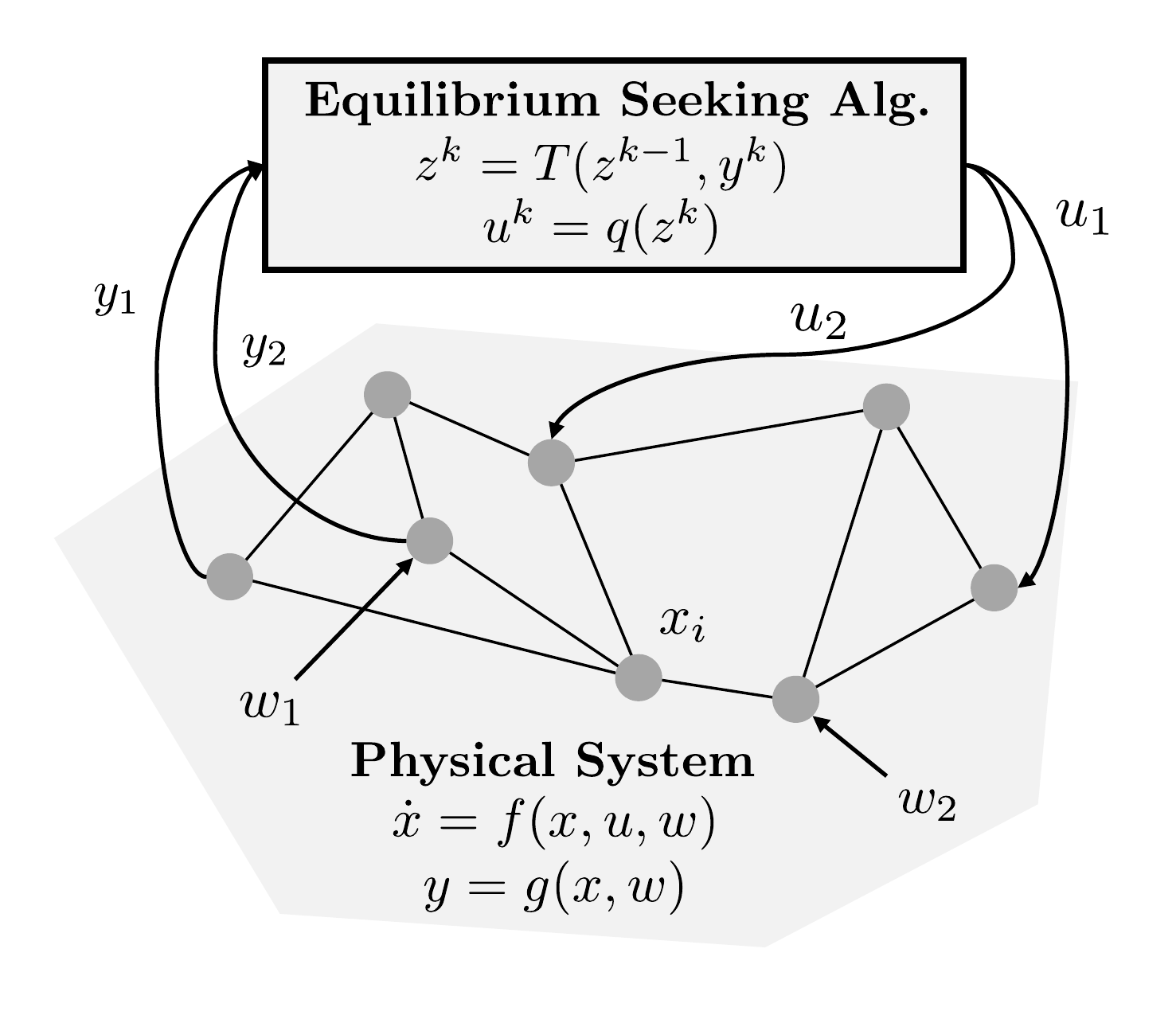}
    \caption{In feedback equilibrium seeking, measurements from a countinuous-time dynamical system are incorporated into a discrete-time equilibrium seeking algorithm resulting in a coupled sampled-data cyber-physical system.}
    \label{fig:block-diagram}
\end{figure}

\section{Preliminaries}
\label{sec:prel}
\textit{Basic notation:} Denote by $I$ and $\id	$ the identity matrix and operator, respectively; by $\overline{\lim}$ and $\esssup$ the limit and essential suprema, respectively. Given $P\!=\!P^\top \succ 0$ and $x\in \reals^n$, $|x|_P := \sqrt{x^\top P x}$ with the convention $|x| \!= \!|x|_I$. 

\textit{Systems Theory:} Continuous-time signals are denoted by $\mathbb{L}^n = \{f: \reals_{\geq0} \to \reals^n\}$ and sequences by $\ell^n = \{f: \ints_{\geq 0} \to \reals^n\}$. Given $x\in \sig^n$ and $y\in \ell^n$, we consider the signal/sequence norms $\|x\| = \esssup_{t\geq 0} |x(t)|$ and $\|y\| = \sup_{k\geq 0} |y_k|$. A function $\gamma:\reals_{\geq 0} \rightarrow \reals_{\geq 0} $ is of class $\mc K$ if it is continuous, strictly increasing, and satisfies $\gamma(0)=0$. If it is also unbounded, then $\gamma \in \Kinf$. Similarly, a function $\sigma: \reals_{\geq0} \to \reals_{\geq0}$ is said to be of class $\L$ if it is continuous, strictly decreasing, and satisfies $\sigma(s) \to 0$ as $s\to \infty$. A function $\beta:\reals_{\geq 0} \times \reals_{\geq 0} \to \reals_{\geq 0}$ is of class $\KL$ if $\beta(\cdot,s) \in \K$ for each fixed $s\geq 0$ and $\beta(r,\cdot)\in \L$ for fixed $r$. Class $\mathcal K$ functions obey a weak triangle inequality\cite[Lemma 10]{kellett2014compendium}:
\begin{equation} \label{eq:weak_triangle_ineq}
	\forall a,b\geq 0 \text{ and }\alpha \in \K,~~\alpha(a+b) \leq \alpha(2a) + \alpha(2b).
\end{equation}
For $\gamma_1,\gamma_2 : \reals \to \reals$, $\gamma_1 \circ\gamma_2$ denotes the composition. Given $\xi\in \reals^n$, $u\in \sig^m$ and $f:\reals^n \times \reals^m \to \reals^n$, $x\in \sig^n$ is a solution of the initial value problem
\begin{equation} \label{eq:ctime_intro}
	\dot x(t) = f(x(t),u(t)), \quad x(0) = \xi.
\end{equation}
If it is uniformly continuous, obeys the initial condition and satisfies the differential equation almost everywhere.

\begin{dfn}[LISpS \cite{sontag1996new}] \label{def:LISS}
Consider a system
\begin{equation} \label{eq:iss_sys}
	x'(t) = f(x(t),u(t)),~~ x(0) = \xi
\end{equation}
with solution $x(t,\xi,u)$ (where either $x'(t) = \dot x(t)$ and $t\in \reals_{\geq0}$ or $x'(t) = x(t+1)$ and $t\in \ints_{\geq0}$) and let $\bar x$ be a reference signal. The system \eqref{eq:iss_sys} is said to be Locally Input-to-State Practically Stable (LISpS) about $\bar x$ with respect to $u$ if there exists $\eps> 0$, $\beta \in \KL$, $\gamma \in \K$, and $b> 0$ such that,
\begin{equation*}
	\forall t \geq 0, \quad |x(t,\xi,u) - \bar x(t)| \leq \beta(|\xi - \bar x(0)|,t) + \gamma(\|u\|) + b,
\end{equation*}
provided $|\xi| \leq \eps$ and $\|u\| \leq \eps$. If $b = 0$ then \eqref{eq:iss_sys} is Locally Input-to-State Stable (LISS) about $\bar x$ with respect to $u$.
\end{dfn}

\smallskip
\textit{Operator Theory \cite{bauschke2011convex}:} Given a closed convex set $\Omega \subseteq \reals^n$, $\iota_{\Omega}:\reals^n \to \{0,\infty\}$ denotes its indicator function, $\nc_{ \Omega}:\Omega \mto \reals^n$ is its normal cone operator, and $\proj_{\Omega}:\reals^n \to \Omega$ is the Euclidean projection onto $\Omega$. A set-valued mapping $\mc F:\reals^n \mto \reals^n$ is $\mu$-strongly monotone, if $(u-v)^\top (x-y) \geq \mu \left\| x-y \right\|^2$ for all $x\! \neq \!y \in \reals^n$, $u \in \mc F (x)$, $v \in \mc F (y)$, and monotone if $\mu\!=\! 0$; $\fix(\mc F)=\{x \in \mathbb{R}^n \,|\, x \in \mc F(x)\}$ and $\zer(\mc F) = \{x \in \mathbb{R}^n \,|\, 0 \in \mc F(x)\}$ denote the set of fixed points and of zeros of $\mc F$. For a convex function $f:\reals^n \to \reals$, $\partial f:\reals^n \mto \reals^n$ denotes the subdifferential mapping in the sense of convex analysis.

\begin{dfn}[Strong Regularity\cite{robinson1980strongly}]
\label{def:SR}
A set-valued mapping $\Psi:\reals^n \rightrightarrows \reals^n$ is said to be strongly regular at $(x,y)$ if and only if $y \in \Psi(x)$ and there exist neighborhoods $U$ of $x$ and $V$ of $y$ such that the truncated inverse mapping $\tilde{\Psi}^{-1}: V\mapsto \Psi^{-1}(V)\cap U$ is a Lipschitz continuous function on $V$.
\end{dfn}


\section{Problem Setting}
\label{sec:ProbSett}
\blue{We consider the problem of efficiently operating a physical plant described by the following nonlinear state-space system}%
\begin{subequations} \label{eq:ctime-dynamics}
\begin{align}
\dot{x}(t) &= f(x(t),u(t),w(t)),\\
y(t) &= g(x(t),w(t))
\end{align}
\end{subequations}
where $x\in \sig^{n_x}$ is the state, $y\in \sig^{n_y}$ is the output, $u\in \sig^{n_u} $ is the control input, with $u(t) \in \mc U$ for all $t \in \mathbb{R}_{\geq 0}$, and $w\in \sig^{n_w}$ is a disturbance satisfying $w(t) \in \mc{W}$ for all $t \in \mathbb{R}_{\geq 0}$.

We adopt the ``stabilize-then-optimize'' paradigm, and assume that \eqref{eq:ctime-dynamics} is stable\footnote{\blue{If the physical plant \eqref{eq:ctime-dynamics} is not stable, it can be pre-stabilized by replacing $f(x,u,w)$ with $f(x,k(x,v),w)$, where $v$ is the new input (for example, a reference command) and $k$ is a stabilizing controller.}} and has a steady-state map $p: \mc{U}\times\mc{W}\to \reals^{n_x}$ satisfying $f(p(u,w),u,w) = 0$ for all $u \in \mc{U}, w\in \mc{W}$ and a steady-state input-output map
\begin{equation} \label{eq:IOSS}
	h(u,w) = g(p(u,w),w).
\end{equation}

Formally, we assume the system \eqref{eq:ctime-dynamics} satisfies the following properties that ensure its solution trajectories and steady-state mappings are well defined.
\begin{ass}[Robust Stability of the Plant \eqref{eq:ctime-dynamics} \blue{with respect to parameter variations}] \label{ass:system_properties}
(i) $f$ is locally Lipschitz; (ii) $g$ is globally $L_g$-Lipschitz; (iii) $w$ is continuously differentiable, and $\dot{w}\in \sig^{n_w}$ satisfies $\|\dot{w}\| < \infty$; (iv) $\mc{W}$ and $\mc{U}$ are compact and convex; (v) \eqref{eq:ctime-dynamics} is LISS \cite{sontag1996new}, namely, there exists $\eps_x,\eps_w, \alpha_3,\alpha_4,\alpha_5 > 0$, a continuously differentiable function $V$, and $\sigma_1 \in \K$ such that for any constant $u \in \mc{U}$
\begin{gather*}
	\alpha_3 |x - p(u,w)|^2 \leq V(x,u,w) \leq \alpha_4 |x - p(u,w)|^2,\\
	\dot{V}(x(t),u,w(t)) \leq -\alpha_5 V(x(t),u,w(t)) + \sigma_1(|\dot{w}(t)|),
\end{gather*}
if $V(x(t),u,w(t)) \leq \epsilon_x$ and $|\dot{w}(t)| \leq \epsilon_w$. {\hfill $\square$}
\end{ass}

Our control objective is to design an output feedback controller that drives \eqref{eq:ctime-dynamics} and maintain it near efficient operating conditions. We encode ``efficiency'' using the following structured generalized equation (GE), parameterized by $w \in \mc W$:%
\begin{subequations} \label{eq:GE}
\begin{align}
\label{eq:GE_1}
& \quad 0 \in  G(z,s) + \mc{A}(z) & \text{(efficiency objective)} \\
\label{eq:GE_2}
& \quad s = h(u,w), &\text{(steady-state map)} 
\\
\label{eq:GE_3}
&\quad u = q(z),  &(\text{ctrl state to input map}) 
\end{align}
\end{subequations}
where $z\in \reals^{n_z}$ is an auxiliary variable, $s$ is the steady-state output of \eqref{eq:ctime-dynamics}, $G:\reals^{n_z} \times \reals^{n_y} \rightarrow \reals^{n_z}$ is a single-valued mapping\footnote{For notational simplicity $G$ does not depend on $w$, this is without loss of generality since $w$ can be included in $s$.}, $\mc{A}:\reals^{n_z} \mto \reals^{n_z}$ is a set-valued mapping, and $q:\reals^{n_z} \to \mc{U}$ is the output mapping.
\blue{The notion of efficiency encoded in the GE \eqref{eq:GE} is flexible and encompasses a wide variety of useful objectives, notably constrained optimization (e.g., minimizing energy consumption as described in \ref{sec:SM}) and generalized Nash equilibria.} The auxiliary variable $z$ is the internal state of the controller, and it also allows the modelling of optimality and equilibrium conditions that involve dual variables, such as KKT critical points. Two concrete examples are provided at the end of this section.

\smallskip
Our control objective is to maintain the system \eqref{eq:ctime-dynamics} near solutions of \eqref{eq:GE}; we must impose some regularity conditions to ensure that this is a well-defined problem. Substituting \eqref{eq:GE_2} and \eqref{eq:GE_3} into \eqref{eq:GE_1} yields the following compact GE:
\begin{equation} \label{eq:GE_compact}
	 0 \in\mathbb{G}(z,w) + \mc{A}(z), 
\end{equation}
where $\mathbb{G}(z,w) = {G}(z,h(q(z),w))$.
The parameter-to-solution mapping $S: \mc W \rightarrow \bb R^{n_z}$ is defined as
\begin{equation}
\label{eq:solMapping}
	S(w) = \{z \in \mathbb R^{n_z}~|~  0 \in\mathbb{G}(z,w) + \mc{A}(z)\}.
\end{equation}
For a given $w \in \mathbb L^{n_w}$, the set of solution trajectories is
\begin{equation}
	\mc S(w) = \{z \in \sig^{n_z}~|~\forall t \geq 0,~~z(t) \in S(w(t))\}.
\end{equation}
The following assumption ensures that tracking solution trajectories in $\mc S$ is a well-posed problem.
\begin{ass}[Strong Regularity of the GE] \label{ass:strong_reg}
(i) $q$ is globally $L_q$-Lipschitz continuous, (ii) $G$ is continuously differentiable, and (iii) the mapping $\mathbb{G}(\cdot,w) + \mc{A}(\cdot)$ in \eqref{eq:GE_compact} is strongly regular at all points satisfying $z \in S(w)$, for all $w\in \mc{W}$.
{\hfill $\square$}
\end{ass}
\blue{Conditions for strong regularity are context-dependent, as detailed later in this section with two examples. In optimization contexts, it reduces to having a strong local minimizer and unique dual variables, which follows by second-order sufficient conditions and suitable constraint qualifications.
}

In general, the set $\mc S(w)$ can be complex and multi-valued, thankfully Assumption~\ref{ass:strong_reg} imposes some structure on $\mc S(w)$.

\begin{thm} \label{thm:sol_traj}\cite[Theorem 3.2]{dontchev2013euler} Under Assumption~ \ref{ass:strong_reg}, for any $w \in \mathbb{L}^{n_w}$, there exist $m$ Lipschitz continuous mappings $\bar z_i \in \sig^{n_z}$, $i \in \{1,\ldots,m\}$, such that $\mc S(w) = \{\bar z_1,\ldots,\bar z_m\}$.
{\hfill $\square$}
\end{thm}
\begin{rmk}\label{rmk:branches}
\blue{Theorem~\ref{thm:sol_traj} demonstrates that the solution map consists of $m$ isolated w-parameterized trajectories called ``branches''. These branches generalize the notion of a isolated local solution of a GE (in the context of optimization, a strict local minimizer) to a parameterized setting. Our analysis will eventually establish tracking error bounds and the existence of a region of attraction around each branch.}
\end{rmk}

\noindent We can now formally state our control problem.
\begin{prb} \label{prb:control_problem}
Design an output feedback controller so that the system \eqref{eq:ctime-dynamics} tracks the input and output trajectories\footnote{With an abuse of notation, we extend the mappings $h$ and $q$ to take signals as arguments, in the sense of $u^*(t)=q(z^*(t))$ for all $t \in \mathbb{R}_{\geq 0}$.} $u^* = q(z^*)$ and $y^*=h(u^*,w)$, for some $z^* \in \mc S(w)$ and $w \in \mathbb{L}^{n_w}$.
\end{prb}

The GE \eqref{eq:GE} allows us to naturally express complex control objectives, e.g., constraints can be readily encoded into $\mc{A}$ using normal cone operators (see Example 1.A). Similarly to FO \cite{hauswirth2021optimization}, we assume that the dynamic model of the plant \eqref{eq:ctime-dynamics} is unknown and that the exogenous disturbance $w$ is not measurable. However, the output $y$ can be measured and the input-output steady-state sensitivity $\nabla_u h(u,w)$, or an approximation of it, is available. Robustness of FO controllers to static input-output modelling error is analytically studied in \cite{colombino2019towards} and experimentally in \cite{ortmann2020experimental}.
The inability to measure $w$ is common in practical applications. For example, in power systems, $w$ represents variable micro-generation and uncontrollable loads caused by consumers drawing power from the grid.


\subsection*{\textbf{Example 1.A.} Nonlinear Programming (NLP)} 
%
Consider the $w$-parameterized nonlinear program (NLP)
\begin{subequations} \label{eq:NLP}
\begin{align} 
    \min_{\xi,u}   &\quad  \phi(\xi,u) + \varphi(\xi) \\
    \text{s.t.} & \quad  \xi  =  h(u,w),~~ u \in \mc{U}
\end{align}
\end{subequations}
where \textcolor{black}{$\phi:\reals^{n_y} \times \reals^{n_u}\to\reals$} and $h$ are twice continuously differentiable, and $\varphi:\reals^{n_y} \to \reals\cup \{\infty\}$ is a proper lower semicontinuous convex function \cite[Def.~9.12]{bauschke2011convex}, e.g., an $\ell_1$-norm. We assume that \eqref{eq:NLP} is always feasible, namely,
\begin{equation*}
	\{(\xi,u)~|~\xi\in \dom \varphi, u \in \mc{U}, \xi = h(u,w)\} \neq \varnothing, \quad
	\forall w\in \mc{W}.
\end{equation*}
The associated partial Lagrangian function is
\begin{equation} \label{eq:NLP-lagrangian}
L(\xi,u,\lambda,w) = \phi(\xi,u) +\lambda^{\top}(h(u,w) - \xi),
\end{equation}
where $\lambda\in \reals_{\geq 0}^{n_y}$ is a dual variable, and the KKTs for \eqref{eq:NLP} are
\begin{equation} \label{eq:NLP_KKT0}
\begin{cases}
0 \in	\nabla_\xi L(\xi,u,\lambda,w) +\partial \varphi(\xi)\\
0 \in \nabla_u L(\xi,u,\lambda,w) + \mc{N}_\mc{U}(u)\\
0 =	\nabla_\lambda L(\xi,u,\lambda,w) = h(u,w) - \xi
\end{cases}
\end{equation}
This is a special case of the GE in \eqref{eq:GE}, with $z=(\xi,u,\lambda)$,
\begin{equation} \label{eq:NLP_KKT}
0 \in	\underbrace{
\begin{bmatrix}
	\nabla_\xi L(\xi,u,\lambda,w)\\
	\nabla_u L(\xi,u,\lambda,w)\\
	s - \xi
	\end{bmatrix}
	}_{G(z,s)}
+
\underbrace{
\begin{bmatrix}
\partial \varphi(\xi)\\
\mc{N}_\mc{U}(u)\\
0
\end{bmatrix}
}_{\mc A(z)},
\end{equation}
$s=h(u,w)$ and $q(z) = [0 ~~ I ~~ 0] z$.
The following statement provides sufficient conditions for strong regularity of \eqref{eq:NLP_KKT}.
\begin{lmm} \label{lmm:NLP_strong_reg}
Consider any $w\in \mc{W}$, a KKT point $\bar z = (\bar \xi, \bar u,\bar \lambda) \in S(w)$, and let $\zeta = (\xi,u)$. Then, the condensed GE \eqref{eq:GE_compact} associated with \eqref{eq:NLP_KKT} is strongly regular at $\bar z$ if $\zeta^\top \nabla_\zeta^2 L(\bar z,w) \zeta > 0$ for all $\zeta=(\xi,u)$ satisfying $\xi = \nabla_u h(\bar u, w) u$.
\end{lmm}
\begin{proof}
The proof is identical to that of \cite[Theorem 7]{liao2020time}.
\end{proof}

The condition in Lemma~\ref{lmm:NLP_strong_reg} is known as a second order sufficient condition and can be checked numerically\footnote{\blue{Specifically, if $Z$ is a basis for the nullspace of $[I~~-\nabla_u h(\bar u,w)]$ then the regularity condition is equivalent to $Z^T \nabla_\zeta^2 L(\bar z,w) Z \succ 0$ which can be checked using e.g., a Cholesky factorization.}} if $\bar z$ is known. Conditions of this type are standard in the analysis of Newton's method and sequential quadratic programming, see e.g., \cite[Theorem 1.13, Theorem 4.14]{izmailov2014newton}.
{\hfill $\blacktriangle$}

\subsection*{\textbf{Example 2.A.} Generalized Nash Equilibrium Problems}
Consider a set of $N$ agents labelled by $i \in \mc I \!:=\! \{1,\ldots,N \}$. Each agent $i \in \mc I$ chooses a control action $u_i$ from a polyhedral set $\mc U_i := \{\xi \in \reals^{n_i}~|~B_i \xi \leq b_i\}$, and has a cost function $J_i(u_i,s)$ that depends on its own action, $u_i$, and the actions of other agents, $u_{-i}=(u_1,\ldots,u_{i-1},u_{i+1},\ldots,u_N)$, through the steady state of the physical system $s=h(u,w)$, where $u=(u_1,\ldots,u_N)$. Moreover, the actions of all agents are linked via affine coupling constraints of the form $Au \leq b$, modelling limited availability of shared resources \cite{belgioioso2022distributed}. 
The agents are self-interested and want to optimize their own local cost. The resulting collection of the $N$ parameterized inter-dependent optimization problems constitutes a parametrized game:
\begin{align} \label{eq:Game}
\forall i \in \mc I: \ \left\{
\begin{array}{r l} \displaystyle
  \min_{u_i }   & J_i(u_i,h(u,w))=: \mathbb{J}_i(u_i,u_{-i},w)  \\
  \text{s.t.}              & A u \leq b , \ u_i \in \, \mc U_i 
\end{array}
\right.
\end{align} 
From a game-theoretic perspective, a relevant solution concept for \eqref{eq:Game} is the generalized Nash equilibrium, where no agent can unilaterally reduce its cost \cite{facchinei2010generalized}.

\begin{dfn}\label{def:GNE}
A feasible action profile $\bar u = (\bar u_1, \ldots, \bar u_N)$ is a generalized Nash equilibrium (GNE) of the game \eqref{eq:Game} if
\begin{align*}
\mathbb  J_i(\bar u_i,\bar u_{-i},w) \leq  
\min_{u_i \in \mc{U}_i}\left\{
\mathbb J_i(u_i,\bar u_{-i},w) ~|~A(u_i,\bar u_{-i}) \leq b 
\right\},
\end{align*}
holds for all agents $i \in \mc I$.
{\hfill $\square$}
\end{dfn}

If each $\mathbb J_i$ is continuously differentiable and convex with respect to $u_i$, a GNE can be found by solving a system of coupled KKT conditions of the problems \eqref{eq:Game} \cite[Th.~4.8]{facchinei2010generalized}:
\begin{align} 
\label{eq:GNEge}
\begin{cases}
0 \in \nabla_{u_i} \mathbb J_i( u_i, u_{-i},w)   + A_i^\top \lambda + \mathcal{N}_{\mathcal{U}_i}(u_i), \; \forall i \in \mathcal{I}\\
0 \in  b - Au + \mathcal{N}_{\mathbb{R}^m_{\geq 0}}(\lambda),
\end{cases}
\end{align}
where $\lambda \in \mathbb R^m$ is a dual variable associated with the coupling constraints $Au -b \leq 0$. The set of solutions to \eqref{eq:GNEge} is a special subclass of GNEs, known as \textit{variational} GNEs (v-GNEs) \cite{facchinei2010generalized}.

To write \eqref{eq:GNEge} in a compact form, we define the pseudo-gradient $ 
\mathbb F(u,w) = F(u,h(u,w)) = [\nabla_{u_i} J_i(u_i,h(u,w))]_{i \in \mathcal I}$ obtained by stacking up the partial gradients of the local cost functions $J_i$, i.e., $\nabla_{u_i} J_i(u_i,h(u,w)) =: F_i(u,h(u,w))$. The KKT system \eqref{eq:GNEge} can then be cast as a special case of the original GE \eqref{eq:GE} with $z = (u, \lambda)$, $q(z)= [I~~0] z = u$, and
\begin{equation} \label{eq:GEvGNE}
0 \in \underbrace{
\begin{bmatrix}
F(u,s)\\b \end{bmatrix}}_{ G(z,s)}  + 
\underbrace{\begin{bmatrix}
0 & A^\top \\ -A & 0
\end{bmatrix}\begin{bmatrix}
u\\
\lambda
\end{bmatrix} + 
\begin{bmatrix}
\mathcal{N}_{\mc U}(u)
\\
\mc N_{\R_{\geq 0}^m}(\lambda)
\end{bmatrix}
}_{\mc A(z)},
\end{equation}
To ensure the problem is well posed, we assume that the game primitives satisfy the following technical conditions:
\begin{enumerate}[(C1)]
\item  $\mathbb F$ is continuously differentiable and, $\forall w \in \mc W$, $\mathbb  F(\,\cdot\, ,w)$ is $\tilde \mu$-strongly monotone and $\tilde \ell$-Lipschitz continuous;
\item For all $ u \in \mc U$, $F(u, \,\cdot\,)$ is $\ell$-Lipschitz continuous;
\item For all $(\bar u,w)$ and $w\in \mc{W}$ satisfying \eqref{eq:GNEge}, it holds that $[\tilde A_i]_{i\in \mc{E}(\bar u)}$ has full row rank, where $\tilde A = \begin{bmatrix}
		A^{\top} & B_1^{\top}&\ldots & B_N^{\top}
	\end{bmatrix}^{\top}$, $\mc{E}(\bar u) = \{i\in \bb{N}~|~ \tilde A_i \bar u = \tilde b\}$ is the set of active constraints at $\bar u$, $\tilde b = (b,b_1,\ldots,b_N)$, \blue{and $\tilde {A}_i$ is the $i$th row of $\tilde A$.}
\end{enumerate}
\begin{lmm}
\label{lem:gameReg}
Under (C1)--(C3), the condensed GE \eqref{eq:GE_compact} associated with \eqref{eq:GEvGNE} is strongly regular at $(\bar z,w)$, with $\bar z \in S(w)$.
\end{lmm}
\begin{proof}
See Appendix~\ref{ap:Lemma_gameReg}.
\end{proof}
Condition (C3) is standard in optimization and is known as the Linear Independence Constraint Qualification (LICQ). {\hfill $\blacktriangle$}

\section{Control Strategy}
\label{sec:control-strat}
Our objective is to maintain the system \eqref{eq:ctime-dynamics} near efficient operating points, namely, the solution trajectories $s^*=h(u^*,w)$ of the GE in \eqref{eq:GE}. Since \eqref{eq:ctime-dynamics} is pre-stabilized, selecting $u(t) = u^*(t)$ for all $t\geq 0$ would cause \eqref{eq:ctime-dynamics} to approximately track the desired steady-state $s^*(t)$. However, computing $u^*(t)$ requires full knowledge of $w(t)$ and evaluating the solution mapping $S(w(t))$ which may be impossible and/or impractical. Instead, we approach the problem by modifying an iterative algorithm for solving the GE \eqref{eq:GE} with the following form
\begin{subequations} \label{eq:algo}
\begin{align}
	s^k &= h(q(z^k),w), \label{eq:algo1} \\
	z^{k+1} &= T(z^k,s^k),
	 \label{eq:algo2}
\end{align}
\end{subequations}
where $T:\reals^{n_z}\times \reals^{n_y}\to \reals^{n_z}$ is the \textit{algorithm}\footnote{For notational simplicity $T$ does not depend on $w$ directly, this is without loss of generality since $w$ can be included in $s$.}, i.e., a rule for generating the next iterate. This class of algorithms is abstract and broad, including e.g. projected-gradient, SCP, and \blue{best-response dynamics in strictly convex games}. Concrete examples are given are the end of this section and in \cite{9683614}.

By substituting \eqref{eq:algo1} in \eqref{eq:algo2}, we can compactly cast the algorithmic update rule \eqref{eq:algo} via the condensed parameterized mapping $\mathbb{T}(\cdot,w):\mathbb{R}^{n_z}\rightarrow \mathbb{R}^{n_z}$, defined as
\begin{equation} \label{eq:T_condensed}
	\mathbb{T}(z,w) := T(z ,h(q(z),w)).
\end{equation}
The following ensures that the nominal iteration \eqref{eq:algo} is locally convergent and well-behaved in a parameterized setting. 
\begin{ass}[Robust Convergence of the Algorithm]
\hspace*{1em}
 \label{ass:algo}
\begin{enumerate}[(i)]
\item For all $ w\in \mc{W}$, $z =\mathbb{T}(z,w)$ if and only if $z \in S(w)$;
\item There exist a continuous function $W:\reals^{n_z}\times \mc W \to\reals$, a constant $\eps > 0$, $\alpha_1,\alpha_2 \in \Kinf$, and $\alpha \in \K$ such that for all $w \in \mc W$, $\bar z \in S(w)$, and $z \in \{\xi:\, |\xi-\bar z | \leq \epsilon \}$
\begin{subequations}
\begin{gather}
	\alpha_1(|z- \bar z|)\leq W(z,w) \leq \alpha_2(|z-\bar z|) \label{eq:Wbounds}\\
	W(\mathbb{T}(z,w),w) \leq W(z,w) - \alpha(|z - \bar z|); \label{eq:Wdot}
\end{gather}
\end{subequations}
\item For all fixed $z$, there exists $L_T > 0$ such that 
\begin{equation*}
	|T(z,s) - T(z,s')| \leq L_T |y-y'|, \quad \forall 
	s,s'\in \reals^{n_y}.~\square
\end{equation*}
\end{enumerate}
\end{ass}
The function $W$ in Assumption~\ref{ass:algo} is commonly known as \textit{``merit function"}, and serves as a Lyapunov function for the algorithm dynamics. Typical choices include objective functions and weighted distances to the solution set $S(w)$.

If $w(t)$ were fully measurable and the steady-state input-output mapping $h$ in \eqref{eq:IOSS} perfectly known, then $u^*(t)$ could be computed using \eqref{eq:algo}. Instead, we construct an output feedback controller by replacing the steady-state input-output model $s^k$ in \eqref{eq:algo1} with online measurements $y^k$ obtained from the physical system \eqref{eq:ctime-dynamics}. This creates an \textit{``online" feedback equilibrium seeking} process, where the system is directly integrated into the algorithm, as illustrated in Figure~\ref{fig:block-diagram}. The incorporation of feedback into the algorithm provides a degree of robustness. \blue{In this paper, we formally establish robustness solely against the unmeasured disturbances $w$. Evidence of robustness against other factors, including modelling errors, has been shown in \cite{colombino2019towards} analytically, and in \cite{ortmann2020experimental} empirically.}

Since $\eqref{eq:algo}$ is discrete, we adopt a sampled-data strategy with a zero-order hold. Let $\tau > 0$ denote the sampling period and let $t^k = k\tau$ be the sampling instants. Interconnecting \eqref{eq:ctime-dynamics} with \eqref{eq:algo} yields the following sampled-data closed-loop system:
\begin{subequations} \label{eq:sampled-data-system}
\begin{align} 
\Sigma_1^s:&\begin{cases} \label{eq:hybrid1}
    ~\dot x(t) = f(x(t),u(t),w(t)),\\
    ~y(t) = g(x(t),w(t)).
\end{cases}\\
\Sigma_2^s: &\begin{cases} \label{eq:hybrid2}
~z^{k} = T(z^{k-1},y(t^{k})),\\
    ~u(t) = q(z^k),~~\forall t\in [t^k,t^{k+1}).
\end{cases}
\end{align}
\end{subequations}

Before proceeding to a closed-loop stability analysis, we provide some concrete examples of admissible algorithms.

\subsection*{\textbf{Example 1.B.} The Josephy-Newton Method}
To solve the GE in \eqref{eq:GE_compact}, the Josephy-Newton (JN) method \cite{josephy1979newton} relies on the implicitly defined iterations
\begin{equation} \label{eq:JN_method}
	H(z) (z^+ - z) +\mathbb{G} (z,w) + \mc{A}(z^+) \ni 0,
\end{equation}
where $H:\reals^{n_z} \to \reals^{n_z \times n_z}$ is an invertible approximation of the Jacobian $\nabla_z\mathbb{G}$. The corresponding update rule is
\begin{equation} \label{eq:JN_FES_operator}
	T(z,s) = (H(z) + \mc A)^{-1}(H(z) z - G(z,s)).
\end{equation}

If the mapping $H$ is chosen judiciously, the nominal JN method  satisfies Assumption~\ref{ass:algo}, as formalized next.

\begin{lmm} \label{lmm:JN_method}
Given Assumption~\ref{ass:strong_reg}, let $\bar z\in S$ and suppose $\exists \bar \epsilon > 0$ such that for all $w\in \mc{W}$ and $z$ such that $|z-\bar z(w)| \leq \bar \epsilon$
\begin{enumerate}[(i)]
	\item $\exists \tilde \delta > 0$ such that $|H(z) - \nabla_z \mathbb{G}(\bar z(w),w)| \leq \tilde \delta$;
	\item The map $(H(z) (\cdot) + \mc{A}(\cdot))^{-1}$ is $M$-Lipschitz continuous.
\end{enumerate}
Then, if $M\tilde \delta < 1$ the JN operator satisfies Assumption~\ref{ass:algo} with $W(z,w) = |z - \bar z(w)|$, $\alpha_1 = \alpha_2 = \id$, and $\alpha = (1-\tilde \eta) \id$.
\end{lmm}
\begin{proof}
See Appendix~\ref{app:JN_method}.
\end{proof}



When the GE represents the KKT conditions of an NLP, the JN method is equivalent to sequential quadratic programming \cite[\S~4.2]{izmailov2014newton}, \cite[\S~6C]{dontchev2009implicit}. Specialized to the KKT conditions \eqref{eq:NLP_KKT} of the NLP \eqref{eq:NLP}, evaluating $T$ is equivalent to letting
\begin{equation}
	T(z,s) = (\xi + d_\xi^*, u + d_u^*, \lambda^*),\
\end{equation}
where $(d_\xi^*, d_u^*,\lambda^*)$ is a primal-dual stationary point of the following optimization problem
{ \small
\begin{subequations} \label{eq:JN_CP}
\begin{align}
\min_{d_\xi,d_u}&~\frac12\left[\begin{smallmatrix} 
d_\xi \\ d_u
\end{smallmatrix}\right]^\top \left(\left[\begin{smallmatrix}
	Q & S^\top \\
	S & R
\end{smallmatrix}\right] \left[ \begin{smallmatrix} 
d_\xi \\ d_u
\end{smallmatrix}\right] + \left[\begin{smallmatrix}
	\nabla_\xi\phi(z)\\ 
	\nabla_u \phi(z)
\end{smallmatrix}\right] \right) + \varphi(\xi\!+\!d_\xi)  \\
\mathrm{s.t.}&~ d_\xi = \nabla_u h(u,w) d_u + s - \xi \label{eq:duals_JN_qp}\\
&~ u + d_u \in \mc{U},
\end{align}
\end{subequations}
}%
The matrices $Q$, $S$, and $R$ are positive semidefinte approximations of $\nabla_\xi^2 L(z,w), \nabla_{u\xi}^2 L(z,w),$ and $\nabla_{u}^2 L(z,w)$, respectively, where $L$ is the Lagrangian function in \eqref{eq:NLP-lagrangian}. {\hfill $\blacktriangle$}
\subsection*{\textbf{Example 2.B.} Semi-decentralized GNE seeking}
Consider the non-cooperative game example in Section~\ref{sec:ProbSett}. A GNE can be computed by solving the GE in \eqref{eq:GEvGNE} using the forward-backward operator splitting (FBS) algorithm
\begin{align} \label{eq:pFB}
T(z,s) = {(\id+\Phi^{-1} \mc A)}^{-1} (z - \Phi^{-1}  G(z,s)),
\end{align}
where $\Phi $ is a preconditioning matrix opportunely designed to distribute the computation among the agents \cite{belgioioso2018projected}, i.e.,
\begin{align}
\label{eq:Phi}
\Phi :=
\begin{bmatrix}
\textrm{diag}(\gamma_1^{-1},\ldots,\gamma_N^{-1}) & -A^\top\\
-A & \gamma_c^{-1} I
\end{bmatrix},
\end{align}
whose diagonal entries are the step sizes of the algorithm.

With this choice, the update \eqref{eq:pFB} can be performed in a semi-decentralized way as 
$T(z,s) = (u^+,\lambda^+)$
with
\begin{align*}
u_i^+ &= \proj_{\mc U_i}\left(
u_i - \gamma_i( F_i(u_i,s) + A_i^\top \lambda)  \right), \quad \forall i \in \mc I\\
\lambda^+ &= \proj_{\mathbb R^m_{\geq 0}}\left(  \lambda + \gamma_\text{c} (A(2u^+-u) -b) 
\right).
\end{align*}

Typically, the local action updates $u_i^+$ are performed in parallel by the agents, whereas the dual update $\lambda^+$, associated with the coupling constraints\footnote{
Typically, the matrix $A \in \mathbb{R}^{m \times n}$ is either sparse or ``fat", i.e., $m \ll n$.
} $Au-b \leq 0$, is managed by a central coordinator able to gather the local strategies $u_i^+$'s and broadcast incentive signals (e.g., prices, dual variables) to all the agents\footnote{Note that fully-distributed GNE seeking algorithms are also available that rely only on local communications among neighbouring agents \cite{bianchi2022fast}.} \cite{belgioioso2021semi}, as illustrated in Figure~\ref{fig:FBScontr}.

\smallskip
In the remainder of this example, we derive conditions under which the update rule in \eqref{eq:pFB} satisfies Assumption~\ref{ass:algo}, and thus it is a suitable controller choice for the FES scheme \eqref{eq:sampled-data-system}.
In the offline static case, namely, when $s=h(z,w)$, the convergence analysis of \eqref{eq:pFB} is based on the theory of \textit{averaged} operators \cite[\S~4.5]{bauschke2011convex}, which are a special subclass of \textit{strongly nonexpansive operators} \cite[Prop.~4.35]{bauschke2011convex}.

\begin{dfn}
\label{def:SQNE}
Let $\mathcal D \subseteq \mathbb R^n$ and $ R: \mathcal D \rightarrow \mathbb R^n$. Then, $R$ is \textit{strongly quasi-nonexpansive} (SQNE) if there exist a norm $|\cdot|_P$ and a constant $\rho >0$ such that
\begin{equation} \label{eq:sQNE}
	|R(z) - \bar z|_P^2 \leq |z - \bar z|_P^2 - \rho |R(z) - z|_P^2.
\end{equation}
for all $z \in \mathcal{D}$, for all $\bar z \in \mathrm{fix}\,R$.
{\hfill $\square$}
\end{dfn}

The key result is presented in the next statement, where we show that any parameterized continuous SQNE operator with a unique fixed point admits a merit function as in Assumption~\ref{ass:algo}. 
\begin{prp}
\label{lmm:SQNE}
Let $\mathbb{T}(\cdot,w)$ be SQNE and continuous, and $\mathrm{fix}~\mathbb{T}(\cdot,w)=\{\bar z(w)\}$, for all $ w \in \mc W$. Then, there exist $P\succ 0$, $\alpha \in \mathcal K$, and $\epsilon>0$ such that for all $ z \in \{\xi:\, |\xi-\bar z | \leq \epsilon \}$
\begin{align}
\label{eq:sQNE-merit-2}
{|\mathbb{T}(z,w) - \bar z(w)|}_P^2 &\leq {|z-\bar z(w)|}_P^2 - \alpha(|z-\bar z(w)|).
\end{align}
\end{prp}
\begin{proof}
See Appendix~\ref{ap:SQNE}.
\end{proof}

\begin{figure}[t]
    \centering
    \includegraphics[width=\columnwidth]{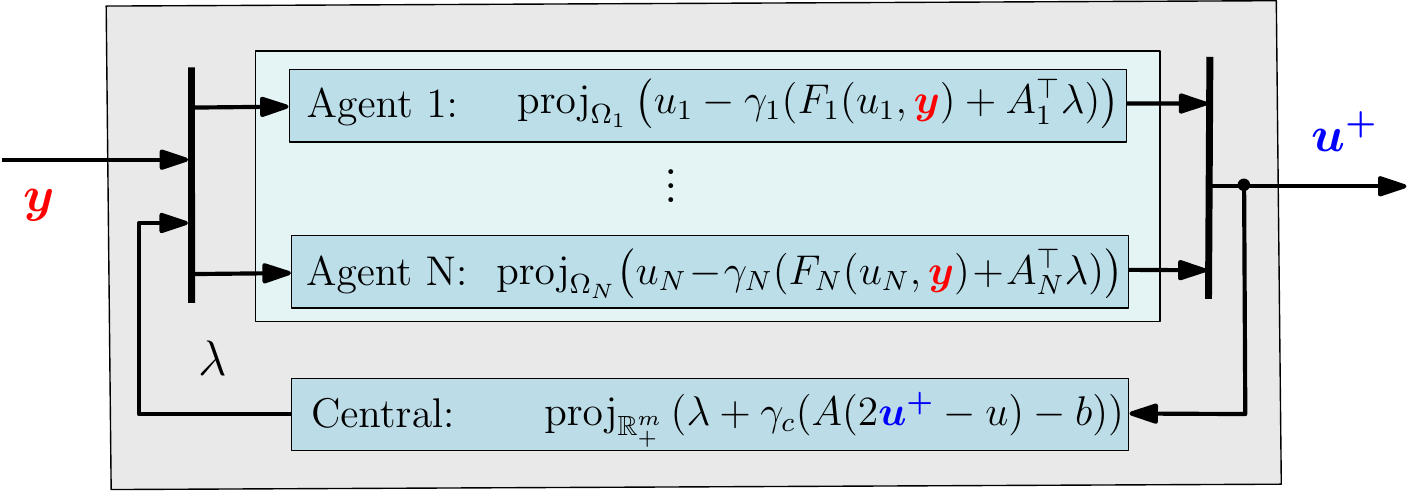}
    \caption{FBS controller. The GNE seeking algorithm \eqref{eq:pFB} as a dynamic (semi) decentralized feedback controller with tunable gains $\gamma_i$'s and state $z = (u,\lambda)$.}
    \label{fig:FBScontr}
\end{figure}

Using Proposition~\ref{lmm:SQNE} we can show that algorithm \eqref{eq:pFB} satisfies Assumption~\ref{ass:algo} and is thus a viable FES controller.
\begin{lmm}
\label{lem:FBS} 
Consider the generalized game in Section~\ref{sec:ProbSett} and the GNE seeking algorithm \eqref{eq:pFB}.
Assume that (C1)--(C3) hold and set the step sizes/gains in \eqref{eq:Phi} such that
$\gamma_i \leq  (\|A_i\|+\delta)^{-1}$, for all $i \in \mc I$, and
$\gamma_c \leq (\sum_{i \in \mc I} \|A_i\|+ \delta)^{-1}$, with $\delta  >  \tilde{\ell}^2/(2 \tilde \mu)$.
Then, the algorithm \eqref{eq:pFB}  satisfies Assumption \ref{ass:algo}.
\end{lmm}
\begin{proof} 
See Appendix~\ref{ap:FBS}.
\end{proof}
%

Algorithms based on strongly quasi-nonexpansive operators \cite[\S~4]{bauschke2011convex} include operator splitting methods (such as forward-backward \cite[\S~26.5]{bauschke2011convex}, Douglas--Rachford \cite[\S~26.3]{bauschke2011convex}
) and other averaged operator-based methods (such as proximal-point algorithm\cite[Th.~23.41]{bauschke2011convex}, alternating projection \cite[Ex.~28.11]{bauschke2011convex}). By Proposition~\ref{lmm:SQNE}, all these algorithms are valid controllers to form the closed-loop interconnection \eqref{eq:sampled-data-system}. 
{\hfill $\blacktriangle$}

\section{Closed-Loop Stability Analysis}
\label{sec:Ana}
We analyze the closed-loop system \eqref{eq:sampled-data-system} by sampling it and forming a discrete-time system, demonstrating LISS of the discretized system, then concluding LISS of \eqref{eq:sampled-data-system} by invoking \cite[Theorem 5]{nevsic1999formulas}. Periodically sampling \eqref{eq:sampled-data-system} yields
\begin{subequations}  \label{eq:discrete-time-system}
\begin{align}
    \Sigma_1^d:&\begin{cases}   \label{eq:sys2d}
    ~ x^{k+1} = \psi^k(x^k,u^{k},\tilde w^k),\\
    ~ ~~~y^k = g(x^k,w^k),
    \end{cases}\\
\Sigma_2^d: &\begin{cases} \label{eq:sys1d}
    ~ z^{k+1} = T(z^k,y^{k+1}),\\
    ~ ~~~u^k = q(z^k)
    \end{cases}
\end{align}
\end{subequations}
where $\psi^k(\xi^k,v,\tilde w^k)$ is the solution of the initial value problem
\begin{equation}
	\dot{\xi}(t) = f(\xi(t),v,w(t)),~~\xi(t^k) = \xi^k,~~t\in [t^k~t^{k+1}]
\end{equation}
at time $t = t^{k+1}$ and $\tilde w^k$ is the restriction of $w$ to $[t^k,t^{k+1})$.

\blue{As noted in Remark~\ref{rmk:branches}, the solution of the GE \eqref{eq:GE_compact} consists of finitely many isolated branches. Here, we establish LISpS for each branch, thereby defining their respective regions of attraction. The closed-loop system \eqref{eq:sampled-data-system} will then track a particular branch, determined implicitly by the initial condition $z(0)$, akin to how an initial condition in gradient-based non-convex optimization implicitly selects a specific local minimizer.}

\blue{For the sake of analysis, consider an arbitrary branch $\bar z \in S$ with corresponding solution trajectory $z^* = \bar z(w) \in \mc{S}$.} We introduce the error signals
\begin{gather*}
	 \delta x(t) = x(t) - p(u(t),w(t)) \text{ and } e(t) = z(t) - z^{*}(t)
\end{gather*}
and shift \eqref{eq:discrete-time-system} to the origin to obtain the error dynamics
\begin{subequations}  \label{eq:discrete-time-error-system}
\begin{align}
    \Sigma_1^e:&\begin{cases}   \label{eq:sys1e}
    ~ \delta x^{k+1} = \mc{G}_1(x^k,u^{k},\Delta u^k,\tilde w^k)
    \end{cases}\\
\Sigma_2^e: &\begin{cases} \label{eq:sys2e}
    ~ e^{k+1} = \mc{G}_2(e^k,\delta x^{k},\tilde w^k),\\
    ~ \Delta u^k = \mc{H}(e^k,\delta x^k,\tilde w^k),
    \end{cases}
\end{align}
\end{subequations}
where $\Delta u^k = u^{k+1} - u^k$, $p$ is the steady-state mapping of the discretized and continuous-time systems, and the expressions for $\mc{G}_1,\mc{G}_2$ and $\mc{H}$ are given in Appendix~\ref{app:expressions}. The resulting feedback interconnection is illustrated in Figure~\ref{fig:iss_diagram}.
\begin{figure}[htbp]
	\centering	\includegraphics[width=0.9\columnwidth]{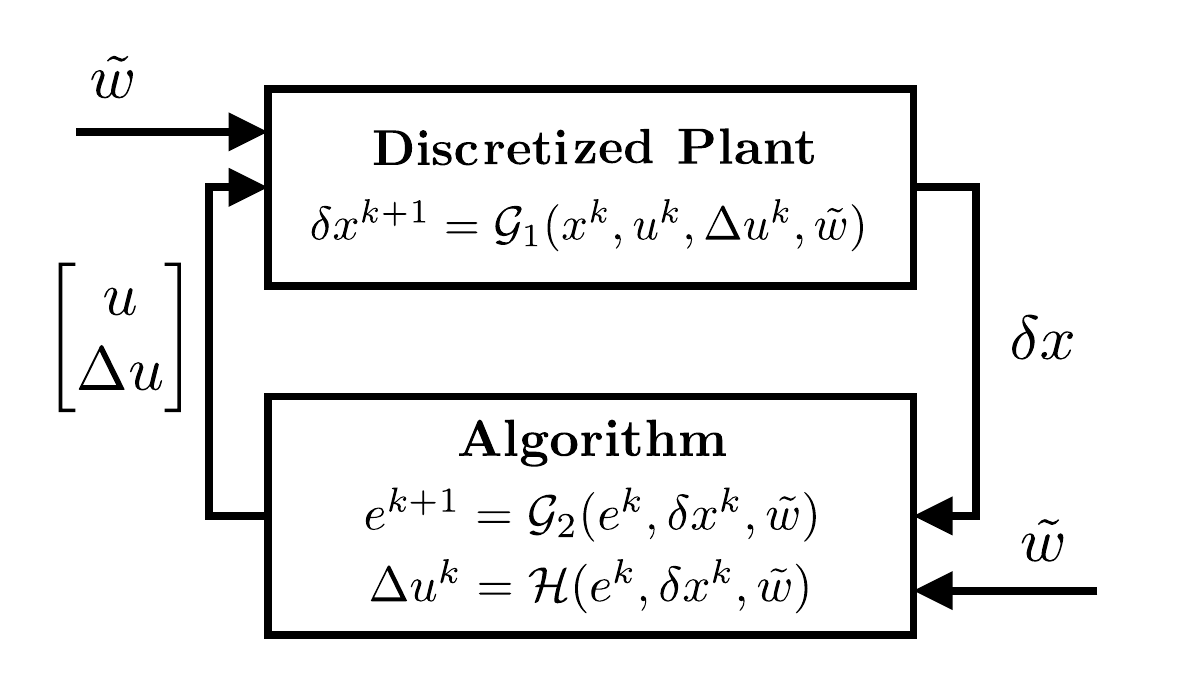}
	\caption{The discrete-time interconnection \eqref{eq:discrete-time-system} can be reconfigured into error coordinates leading to a feedback interconnection of two systems represented by $\mc G_1$, $\mc G_2$, and $\mc{H}$. Expressions for $\mc{G}_1,\mc{G}_2$ and $\mc{H}$ are given in Appendix~\ref{app:expressions}.}
	\label{fig:iss_diagram}
\end{figure}

The main step in the analysis is to show that \eqref{eq:discrete-time-error-system} is LISpS with respect to the disturbance sequence
\begin{equation} \label{eq:d_def}
	d^k = \underset{t\in [t^k~t^{k+1}]}{\esssup} |\dot{w}(t)|.
\end{equation}
We first show that \eqref{eq:sys1e} is LISS with respect to $\Delta u$ and $d$ then show that \eqref{eq:sys2e} is LISS and locally input-to-output stable with respect to $\delta x$ and $d$. Then we employ a small-gain theorem to derive conditions for LISpS of the feedback interconnection.

The first step is to show LISS of the plant.
\begin{thm}[LISS of the plant] \label{thm:plant_LISS}
Let Assumptions~\ref{ass:system_properties}--\ref{ass:algo} hold. Then \eqref{eq:sys1e} is LISS with respect to $\Delta u$ and $d$, i.e., there exists $\beta_x\in \KL$, $\gamma_x^u\in \L$, $\gamma_x^d\in \KL$, $\pi_1> 0$, and $\pi_2\in \K$ such that
\begin{equation}
	V^k\leq \beta_x(V^0,k\tau) + \gamma_x^u(\tau)\|\Delta u\| + \gamma_x^d(\|d\|,\tau)
\end{equation}
where $V^k = V(x^k,u^k,w^k)$, provided that $V^0 \leq 0.25 \eps_x$, $\|\Delta u\| \leq \pi_1 \eps_x$, $\pi_2(\|d\|) \leq \eps_x$ and $\|d\|\leq \eps_w$.
\end{thm}
\begin{proof}
See Appendix~\ref{ap:plant_LISS}.
\end{proof}
This result is a straightforward consequence of the local input-to-state stability of the original continuous-time system (Assumption~\ref{ass:system_properties}) and states that the discrete-time system \eqref{eq:sys1d} will converge to a steady-state for any given constant input with the changes in $w$ and $u$ acting as disturbances. The impact of input variations on convergence to the steady-state manifold, quantified by $\gamma_x^u\in \L$, decreases as the sampling period $\tau$ increases,  as control input update rate decreases. 

Next, we show that the equilibrium-seeking algorithm \eqref{eq:algo} when viewed as a dynamical system \eqref{eq:sys2d} is also LISS.


\begin{thm}[LIOS of the algorithm] \label{thm:algo_LIOS}
Let Assumptions~\ref{ass:system_properties}-- \ref{ass:algo} hold. Then the system \eqref{eq:sys2e} is locally input-output stable (LIOS), i.e., there exist $\bar \veps_1, \bar \veps_2,\bar\veps_3 > 0$, $\beta_z,\beta_u \in \KL$, $\gamma_z^x, \gamma_u^x \in \KL$, and $\gamma_z^d,\gamma_u^d\in \K$ such that for all $W^0 \leq \bar \veps_1$, $\|V\| \leq \bar \veps_2$, $\|d\| \leq \bar \veps_3$
\begin{subequations}
\begin{gather}
	W^k\leq \beta_z(W^0,k) + \gamma_z^x(\|V\|,\tau) + \gamma_z^d(\|d\|), \label{eq:algo-lios1}\\
	\|\Delta u^k\| \leq \beta_u(W^0,k) + \gamma_u^x(\|V\|,\tau) + \gamma_u^d(\|d\|), \label{eq:algo-lios2}
\end{gather}
\end{subequations}
where $\|V\| = \sup_{k\geq0} V(x^k,u^k,w^k)$ and $W^k = W(z^k,w^k)$.
\end{thm}
\begin{proof}
See Appendix~\ref{ap:algo_LIOS}.
\end{proof}

In other words, the algorithm \eqref{eq:algo} is not only convergent whenever the plant is at steady-state, as required by Assumption~\ref{ass:algo}, but also robust with respect to transient deviation from this steady-state as measured by the ISS Lyapunov function $V$ defined in Theorem~\ref{thm:plant_LISS}. 

Having established LISS of the sub-systems, we derive conditions under which the interconnection is LISS. 
\begin{thm} \label{thm:dt_coupled_iss}
Let Assumptions~\ref{ass:system_properties}--\ref{ass:algo} hold. Then, the discrete-time closed-loop system \eqref{eq:discrete-time-error-system} is LISS if the small gain condition
\begin{equation} \label{eq:small-gain}
	\gamma_x^u(\tau) \gamma_u^x(s,\tau) < s~~\forall s\in[0,\tilde \eps].
\end{equation}
is satisfied where $\tilde\eps = \min\{0.25 \eps_x, \bar \veps_2\}$, and $\gamma_x^u\in \mc{L}$, $\gamma_u^x\in \KL$, $\eps_x$, and $\bar \veps_2$ are defined in Theorems~\ref{thm:plant_LISS} and \ref{thm:algo_LIOS}.
\end{thm}
\begin{proof}
See Appendix~\ref{ap:dt_coupled_iss}.
\end{proof}
Theorem~\ref{thm:dt_coupled_iss} is a typical small-gain condition, since \eqref{eq:sys1e} and \eqref{eq:sys2e} are both LISS their interconnection is LISS if the interconnection gains are small enough. Both gains $\gamma_x^u$ and $\gamma_u^x$ tend to decrease, i.e., the stability margins of the interconnection grow, as the sampling period $\tau$ increases.

Due to broad range of system and algorithm combinations allowed under Assumptions~\ref{ass:system_properties}--\ref{ass:algo}, the small-gain condition can be difficult to verify in practice. Theorem~\ref{thm:dt_coupled_iss} can be sharpened if the algorithm is locally linearly convergent.

\begin{cor} \label{corr:simplified-gain}
Let Assumption~\ref{ass:system_properties}--\ref{ass:algo} hold. Additionally, assume that the algorithm \eqref{eq:algo} is locally q-linearly convergent with convergence rate $\eta \in (0,1)$, i.e., Assumption \ref{ass:algo} holds with $W(z,w) = |z-\bar z(w)|_P$, $\alpha(s) = (1-\eta) s$, $\alpha_1(s) = \lambda_{\min}(P) s$, and $\alpha_2(s) = \lambda_{\max}(P) s$. Then $\gamma_u^x(s) = c_1 e^{-\alpha_5\tau} s $, $\gamma_x^u(\tau) = \frac{L_V e^{-\alpha_5 \tau}}{1- e^{-\alpha_5 \tau}}$
and \eqref{eq:small-gain} reduces to the condition
\begin{equation} \label{eq:small-gain-linear}
	\frac{c_1 L_V e^{-2\alpha_5 \tau}}{1- e^{-\alpha_5 \tau}} < 1
\end{equation}
where $c_1 = L_q L_T L_g \sqrt{\frac{\alpha_4}{\alpha_3}} \left(1+ \frac{L_z \lambda_{\max}(P)}{(1-\eta) \sqrt{\lambda_{\min}(P)}}\right)$ and $L_V,L_z,L_T$ and $L_g$ are Lipschitz constants for $V,\bar z, T$, and $g$ with respect to $u, w, y,$ and all arguments, respectively.
{\hfill $\square$}
\end{cor}
Note that 
the small-gain condition \eqref{eq:small-gain-linear} can be always satisfied for large enough $\tau$. Continuous and discrete-time analogs of Corollary~\ref{corr:simplified-gain} can be found in \cite{hauswirth2020timescale} and \cite{simpson2021low}, respectively. 

In the general case, the gain $\gamma_u^x$ is nonlinear and one cannot guarantee LISS for sufficiently large finite $\tau$. We can, however, show LISpS with a decaying perturbation term. 

\begin{thm} \label{thm:main_theorem}
Let Assumptions~\ref{ass:system_properties}--\ref{ass:algo} hold. \blue{Then for each branch $\bar z \in S$} there exists $\bar \tau \in(0,\infty)$  such that for $\tau > \bar \tau$ the sampled-data closed-loop system \eqref{eq:sampled-data-system} is locally input-to-state practically stable (LISpS), i.e., there exist $\beta\in \KL$, $\gamma\in \K$, $\nu_1,\nu_2,\nu_3 > 0$, and $b:\reals_{>0}\to \reals_{\geq 0}$ satisfying $\lim_{\tau\to\infty} b(\tau) = 0$ such that, for all $t\geq 0$
\begin{equation}
	\left|\begin{bmatrix}
		\delta x(t) \\ e(t)
	\end{bmatrix} \right| \leq \beta\left( \left|\begin{bmatrix}
		\delta x(0) \\ e(0)
	\end{bmatrix} \right|,t\right) + \gamma(\|\dot{w}\|) + b(\tau),
\end{equation}
provided that $|\delta x(0)| \leq \nu_1$, $|e(0)|\leq \nu_2$, and $\|\dot{w}\|\leq \nu_3$.
\end{thm}
\begin{proof}
See Appendix~\ref{ap:main_theorem}.
\end{proof}
Theorem~\ref{thm:main_theorem} is the main result of the paper and states that, if the sampling period $\tau$ is long enough, then the interconnection will become locally practically\footnote{The offset term $b(\tau)$ arises because it is not always possible to upper bound $\gamma_u^x \in \KL$ near the origin as required by small-gain theorems. E.g., $\eta(s,\tau) = e^{-\tau} \sqrt{s}$, which does not satisfy the small-gain condition $\eta(s,\tau) < s$ in an neighbourhood of the origin no matter how large $\tau$ is. However, the interval where it does not hold can be made arbitrarily small, i.e., $\eta(s,\tau) < s$ for all $s\in [\eps(\tau),\infty)$ where $\eps(\tau)\to0$ as $\tau\to\infty$. The term $b(\tau)$ is a consequence of this interval and also becomes arbitrarily small. Further, $b$ goes exactly to zero in many cases, e.g, in Corollary~\ref{corr:simplified-gain} or if $\gamma_u^x$ is convex.} stable\footnote{\blue{The constants $\bar \tau, \nu_1,\nu_2,\nu_3$ and functions $b,\gamma,\beta$ differ for each branch.}}. This makes intuitive sense, since the algorithm is guaranteed to converge if the system remains at steady-state (Assumption~\ref{ass:algo}) which longer sampling periods give time to reach (Assumption~\ref{ass:system_properties}). Input constraints can be satisfied pointwise-in-time via projection operations while all other (e.g., output) constraints are guaranteed to be satisfied asymptotically with bounded violations during transients. 

\begin{rmk}
Theorem~\ref{thm:main_theorem} extends our previous result \cite{9683614} to allow local stability/convergence and is the first stability result for sampled-data FO/FES that admits non-linearly convergent algorithms. The latter is especially significant as constrained or distributed  problems are often solved using primal-dual algorithms which are typically sub-linearly convergent.
\end{rmk}



\section{Illustrative Examples}
In this section, we demonstrate that some of the algorithmic preconditions (Assumption \ref{ass:algo}) in Theorem~\ref{thm:main_theorem} are sharp, in the sense that if they are not satisfied, in general, one cannot expect the algorithm-plant interconnection \eqref{eq:sampled-data-system} to be robust to unmeasured disturbances or even stable. The code for the examples in this section is available in~\cite{gitlab_repo}.

Consider a single-input single-output dynamic plant governed by the second-order differential equation
\begin{equation}
\label{eq:DI}
\ddot{\xi} + 0.5 \dot{\xi} +\xi = u + w
\end{equation}
where $y=\xi$ and $w$ is a disturbance term. The plant is an LTI system of the form $\dot{x} = Ax + B(u+w)$ with $x = (\xi,\dot \xi)$ and is asymptotically stable and satisfies Assumption~\ref{ass:system_properties} with steady-state mapping $h(u,w)= u\!+\!w$, Lyapunov function $V(x,u,w) = 1/2|x- (u+w,0)|_P^2$, where $P$ satisfies $AP + PA^\top + I = 0$, and $\alpha_3 = \lambda_{\min}(P), \alpha_4 = \lambda_{\max}(P)$, and $\alpha_5 = \lambda_{\min}(P^{-1})$. The control objective is encoded as
\begin{equation} \label{eq:simple_opt}
\min_u {\textstyle \frac{1}{2} } | y - y^{\text{ref}}|^2
\quad
\text{s.t.} \quad y = h(u,w), \; u \in [-10, \ 10],
\end{equation}
which models set-point regulation and satisfies Assumption~\ref{ass:strong_reg}. 

\subsection{Closing the loop can lead to instability}
The proximal-gradient controller \cite[\S~III.A]{9683614} for this problem generates control inputs according to the update rule
\begin{align}
\label{eq:prox_grad}
T(z^k,y) =
\mathrm{proj}_{[-10,\,10]}
\big(z^k - \gamma(y-y^{\text{ref}})\big),
\end{align}
with $u^k = z^k$, which satisfies Assumption~\ref{ass:algo} if $\gamma \in (0,1]$, with $\alpha_1 = \alpha_2 = \id$, $W(z,w)= |z-z^*(w)|$, $z^*(w) = y^{\text{ref}} - w$, $\alpha(s) = \left(1-\sqrt{ 1-\gamma(2-\gamma) }\right)s$, and $L_T = \gamma$.

Based on Theorem~\ref{thm:main_theorem}, we expect that the interconnection of \eqref{eq:DI} and \eqref{eq:prox_grad} will only be stable for sufficiently large $\tau$. Since the system \eqref{eq:DI} is exponentially stable and the algorithm \eqref{eq:prox_grad} linearly convergent we can use Corollary~\ref{corr:simplified-gain} to conclude that the interconnection is stable for $\tau > 5.44$. The simulation results match our expectations. Figure~\ref{fig:insta} demonstrates that the interconnection indeed becomes unstable for small $\tau$, that the system is stable for sufficiently large $\tau$, and that Corollary~\ref{corr:simplified-gain} can be conservative as $\tau = 5$ leads to a stable interconnection.

%
\begin{figure}[t]
\centering
\includegraphics[width=\columnwidth]{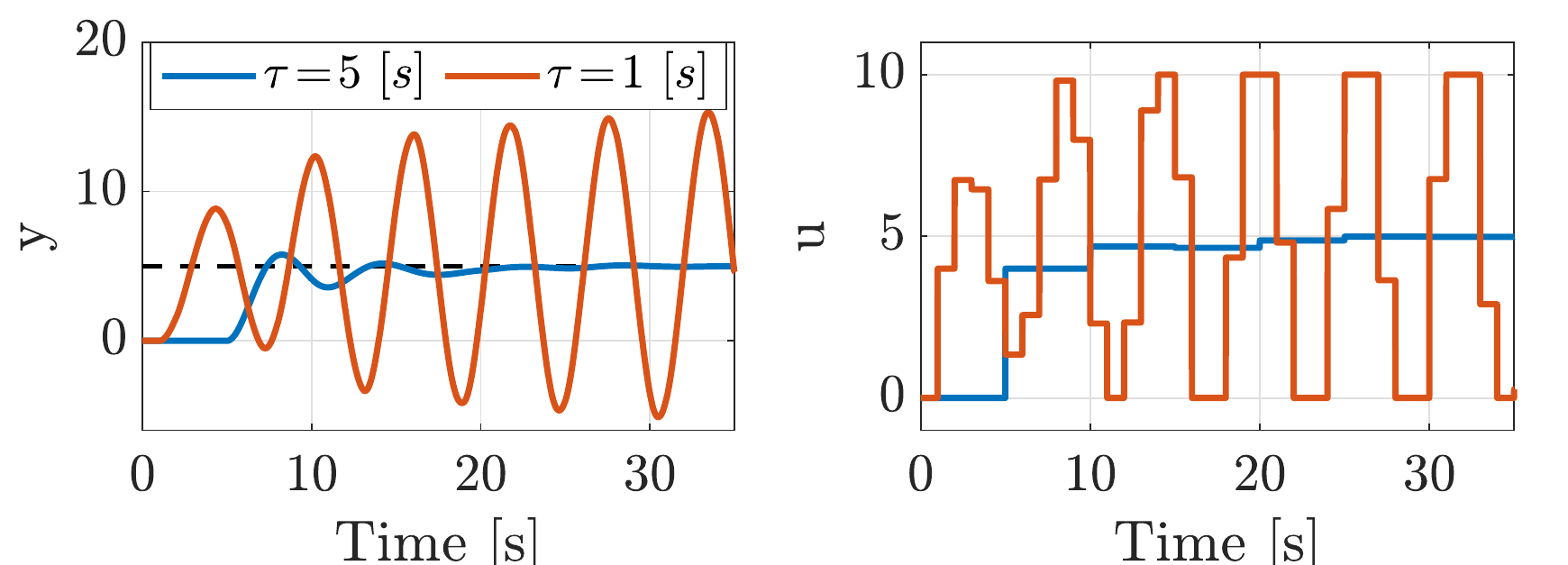}
\caption{Simulations of the sampled-data interconnection of the continuous-time SISO plant \eqref{eq:DI} and the discrete-time algorithm \eqref{eq:prox_grad}, with $\gamma=0.8$, under different choices of the sampling period $\tau$. On the left, the generated output; on the right, the correspondent control input trajectory. 
}
\label{fig:insta}
\end{figure}
%
%
\begin{figure}[t]
\centering
\includegraphics[width = \columnwidth]{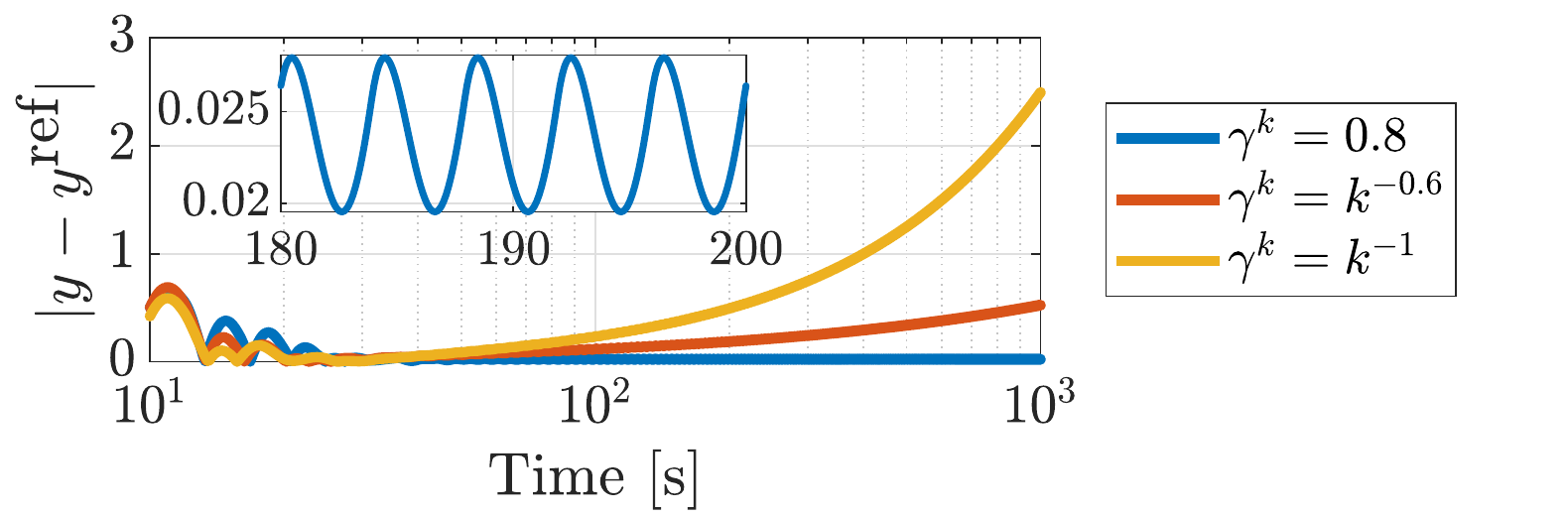}
\caption{Simulations of the sampled-data interconnection of the continuous-time SISO plant \eqref{eq:DI}, subject to the additive disturbance $w(t) \!=\!  5\!\cdot\!10^{-2}t$, and the discrete-time algorithm \eqref{eq:prox_grad}, under different choices of the step size sequence $\gamma^k$. The output tracking error diverges if the sequence vanishes as in \eqref{eq:VSS}, while it stabilizes for constant step sizes. The oscillations are due to the combined effect of time-varying disturbance and sampled-data control.}
\label{fig:VSS}
\end{figure}

\subsection{Vanishing steps do not track solutions trajectories}
Finally, consider the algorithm \eqref{eq:prox_grad}, where the step size is time-varying ($\gamma$ depends on the iteration $k$) and such that
\begin{equation}
\label{eq:VSS}
\textstyle
\gamma^k \geq 0 \quad \forall k, \quad
\sum_{k \in \mathbb N} \gamma^k = \infty,\quad
\sum_{k \in \mathbb N} (\gamma^k)^{2}  < \infty.
\end{equation}
Gradient-based algorithms with vanishing steps of this kind are popular in the context of
optimization and game theory, e.g., for stochastic approximation and gradient tracking. 
Here, a vanishing step size results in a vanishing control gain, and thus unsurprisingly this class of algorithms does not admit a merit function\footnote{
In fact, for $k\rightarrow \infty$ the $\mc K$-function of \eqref{eq:prox_grad} yields $\lim_{k \rightarrow \infty}\alpha_k(\cdot) = \left(1-\sqrt{ 1-\gamma^k(2-\gamma^k) }\right)\cdot= 0$ since  $\lim_{k \rightarrow \infty} \gamma^k = 0$ by \eqref{eq:VSS}.} in the sense of Assumption~\ref{ass:algo}. Hence, ISS cannot be guaranteed in online settings. This is illustrated in Figure~\ref{fig:VSS} which shows that whenever the step sizes in 
 \eqref{eq:prox_grad} vanish to zero, as in \eqref{eq:VSS}, the tracking error diverges.
%

\section{Application Examples}
\subsection{Temperature Regulation in Smart Buildings}
\label{sec:SM}
In this section, we illustrate how FES can be applied to smart building automation. Consider the 5-room single-story office building in Figure~\ref{fig:building}.
Its dynamics are of the form
 \begin{equation}\label{eq:1-build}
     \dot{x} = Ax + B_u u + B_w w + \sum_{i=1}^{n_u} \left(B_{wu,i}w + B_{xu,i}x \right) u_{i},
 \end{equation}
and are generated using the BRCM toolbox~\cite{sturzenegger2014brcm}. The state $x\in \bb{L}^{113}$ contains the temperatures of the rooms and wall layers, floor layers, etc. The control inputs $u\in \bb{L}^{8}$ are an air handling unit (AHU) consisting of air flow (0--1 kg/s), cooling/heating power between $10^2$W and $10^3$W, and one radiator in each room emitting between $0$ and $25$ W/m$^2$. The disturbances $w\in \bb{L}^{10}$ include the solar radiation, the ambient outdoor air temperature, the temperature of the ground, and the internal heat gains coming from building occupants. The measurement $y\in \bb{L}^{7}$ contains the room, outside air, and ground temperatures. The solar radiation and the heat emitted by the buildings occupants are unmeasured. The nonlinearities are caused by the AHU whose control authority depends on the ambient air temperature and room temperatures. 
We model 15 building occupants (providing an internal heat gain of 100W each) by a Markov chain, with a time-dependent probability of being in a given room~\cite{wang2011novel}. 
The solar radiation and ambient temperature are periodic functions yielding temperatures and solar gains representative of central European springtime.
\begin{figure}[t]
  \centering
  \includegraphics[width=.9\columnwidth]{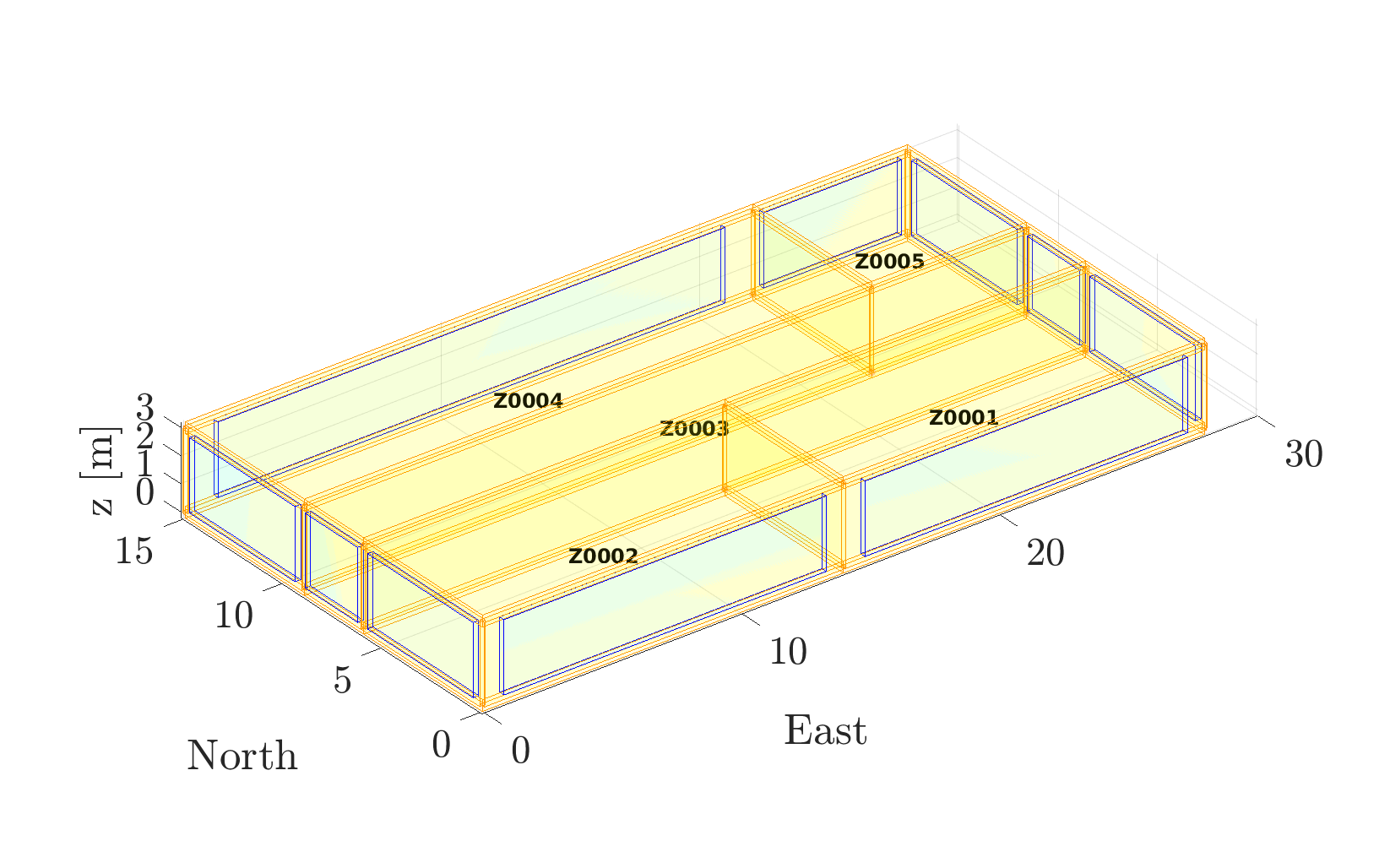}
  \caption{Example building generated via the BRCM toolbox~\protect\cite{sturzenegger2014brcm}.}
  \label{fig:building}
\end{figure}%

Our objective is to minimize energy usage while maintaining the room temperatures within a comfortable range $\mc{T} = [T_{\min}~T_{\max}]$. This control objective is implicitly encoded via an NLP with composite cost function as in \eqref{eq:NLP}, i.e.,
\begin{subequations}\label{eq:building_cost_1}
\begin{align}
&  \phi(\xi,u) =  \frac{\epsilon}{2}\left|
\left[
\begin{smallmatrix} 
\xi \\ u
\end{smallmatrix}
\right]
\right|^2 +  c^\top u, \quad \xi = h(u,w),\\ 
\label{eq:TempCCfunc}
&  \varphi(\xi) = \frac{\eta}{2}\sum_{i=1}^5 \max\{0,\, T_{\min,i} - \xi_i, \, \xi_i - T_{\max,i}\},
  \end{align}
\end{subequations}
where $\eta>0$ is a tuning parameter, $c$ collects the electricity prices, and $T_{\min,i}$, $T_{\max,i}$ are the comfort constraints on the temperature in the $i$-th room.
The quadratic term in $\phi(u,\xi)$ is a regularizer, with typically small tuning parameter $\epsilon$, that improves regularity of the minimizers. 
The purpose of $\varphi$ is to penalize the comfort constraint violations of the room temperatures.
A 1-norm penalty on the violation is used for two reasons. First, the electrical cost of heating is linear in the control input.
Second, $\varphi$ is an exact penalty function, and so for a well-tuned parameter $\eta$, the (disturbance-free) system can be exactly driven within the temperature bounds~\cite{han1979exact}.

Given the control objective \eqref{eq:building_cost_1} and the steady-state sensitivity associated with~\eqref{eq:1-build}, we form a FES controller using the JN algorithm, as explained in Section~\ref{sec:control-strat}.
Given the measurements $y$, the current control input $u$, and the controller state $\xi$, the resulting sampled-data SQP controller sets the next control inputs as the solutions ($d_u$-component) of the QP%
\begin{align}
  \begin{array}{rl}
  \displaystyle
    \min_{d_\xi,d_u,\sigma} & \frac{\epsilon}{2}{\left|
    \left[
    \begin{smallmatrix} 
d_\xi \\ d_u
\end{smallmatrix}
\right] \right|}^2 + \frac{1}{2} d_u^{\top}c + \frac{\eta}{2}\sum_{i=1}^5 \sigma_i\\
    \mathrm{s.t.} & d_\xi = \nabla_u h(u,w) d_u + y-\xi\\
                       & u + d_u \in \mathcal{U}\\
                       & \xi_i + d_{\xi,i} - T_{\max} \leq \sigma_i,
     \quad ~~ \forall i = 1,\dots,5\\
                       & \xi_i + d_{\xi,i} - T_{\min} \geq -\sigma_i,
                       ~~ ~ \forall i = 1,\dots,5\\
                       & \sigma_i \geq 0,
\qquad \qquad \qquad \qquad \forall i = 1,\dots,5
  \end{array}\label{eq:build_jn_step_lp}
\end{align}
where $\sigma_i$ are slack variables that allow to reformulate the nonsmooth cost term $\varphi(\xi+d_{\xi})$ as linear constraints.
\begin{figure}[t]
  \centering
    \includegraphics[width=.9\columnwidth]{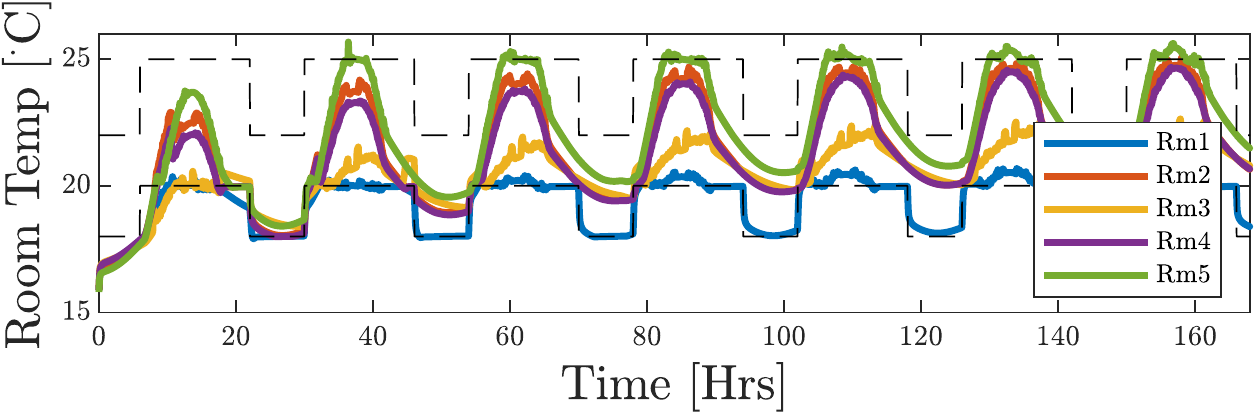}
    \includegraphics[width=.9\columnwidth]{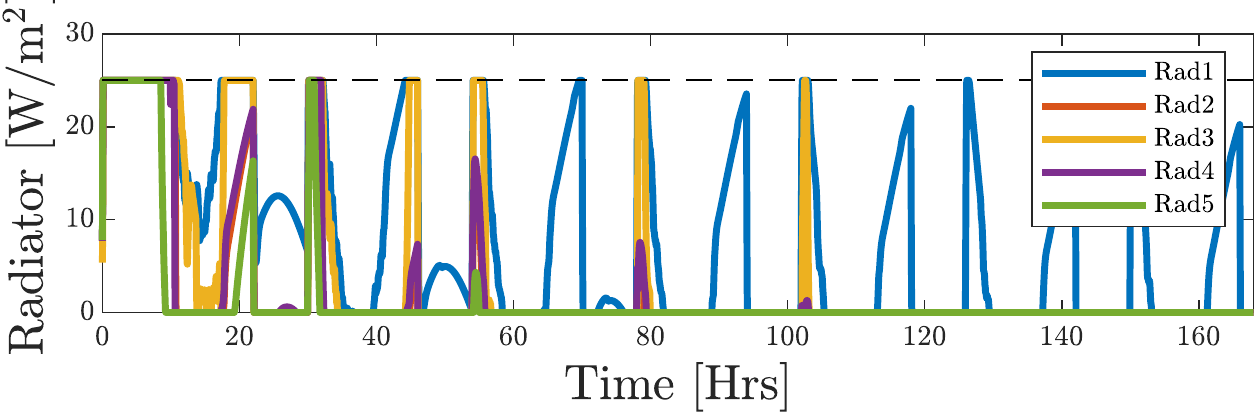}
    \includegraphics[width=.9\columnwidth]{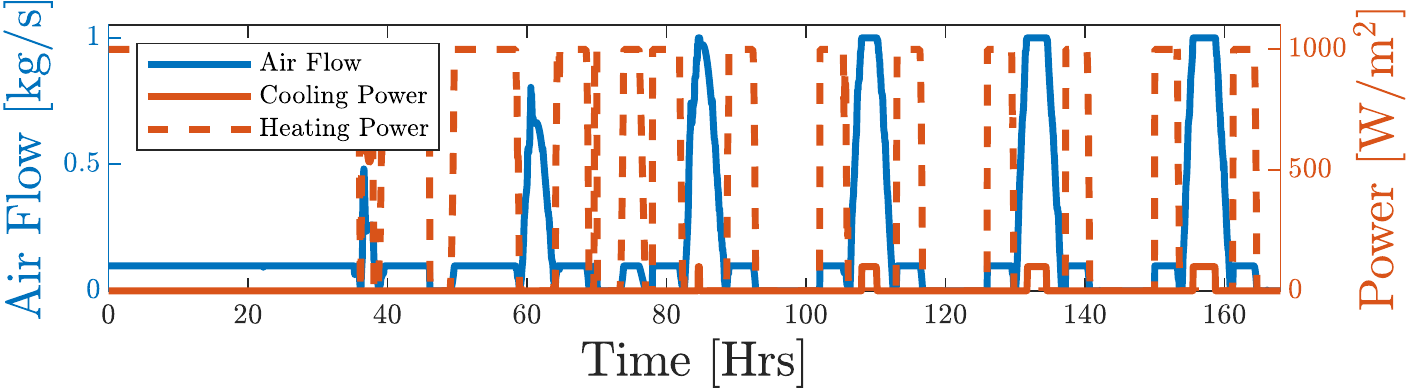}
  \caption{Simulations of the SQP controller \protect\eqref{eq:build_jn_step_lp} on the building dynamics~\eqref{eq:1-build}. 
The comfort temperature (output) constraints are (approximately) satisfied throughout the simulations, while heating and cooling effort is minimized.}
  \label{fig:feedback_opt_results}
\end{figure}
We set $\eta = 5\cdot10^{4}$, $\epsilon = 10^{-5}$, and a discrete time step of $\tau = 3 \  [\text{min}]$.
Simulation results are presented in Figure~\ref{fig:feedback_opt_results}.
The proposed controls are able to keep the rooms between the temperature bounds, with only minor constraint violations.
We compare the controller performance to a hysteresis-based thermostat controller%
, which turns the radiators and AHU heater on when $T_{\text{room}} \leq \frac{T_{\min} + T_{\max}}{2} - 2$ and off when $T_{\text{room}} \geq \frac{T_{\min} + T_{\max}}{2}$. The AHU cooler is turned on when $T_{\text{room}} \geq \frac{T_{\min} + T_{\max}}{2}+2$, and off when $T_{\text{room}} \leq \frac{T_{\min} + T_{\max}}{2}$, with all temperatures in $^\circ$C.
In our example, the SQP controller provides a 27.84\% reduction in constraint violations, and a 32.29\% reduction in total cost as measured by~\eqref{eq:building_cost_1} compared to the hysteresis controller.
The code for this application example is available in~\cite{gitlab_repo}.

\subsection{Competitive Supply Chain Management}
\label{sec:CSC}
Consider a supply chain with $N$ producers $P_i$, labelled by $i \in \mc I :=\{1,\ldots,N\}$, that are supplied by a supplier $S$ and serve a common market $M$, as illustrated in Figure~\ref{fig:supply_chain2}. Each producer orders $o_i$ units of raw materials from the supplier and sells $d_i$ units	 of finished product to the market at price $\sigma_i$. 

Producers are modelled using a simplified version of the model in \cite[Fig.~5a]{spiegler2012control}, where the local stocks of finished product $s_i$ accumulate according to
\begin{equation}
\textstyle
	\dot{s}_i =  l_i - d_i
\end{equation}
with $l_i$ denoting the local production rate. This rate evolves as
\begin{equation} \textstyle
	\dot l_i = -\frac{1}{\tau_i^P}(l_i - o_i)
\end{equation}
where $\tau_i^P$ is the production time constant, and $o_i$ are the orders from the supplier. Each producer maintains a stock of finished product $\bar s_i$ (known as the minimum reasonable inventory \cite{spiegler2012control}) to insulate against shocks using the inventory control law
\begin{equation} \label{eq:lCP}
	o_i = -k_i(s_i - \bar s_i) + d_i,
\end{equation}
where $k_i > 0$ is a tunable control gain.

The common market obeys the linear price/demand curve
\begin{equation} \label{eq:price-curve}
	\bar d_i((\sigma_i,\sigma_{-i}),d_i^\text{w}) = d_i^\text{w} - \beta_i \sigma_i + \textstyle \sum_{i\neq j} \beta_{ij} \sigma_j,
	\quad
	 \forall i \in \mc I
\end{equation}
where $\bar d_i(\sigma,d_i^\text{w})$ is the nominal local demand for producer $P_i$, $d_i^\text{w}$ is a base-line demand, and $\beta_i, \beta_{ij} \geq 0$ are market constants \cite{anderson2010price}. In this model, demand for $P_i$'s product drops if $P_i$ raises its prices $\sigma_i$ and increases when competitors increase theirs $\sigma_{-i}$. The market does not respond instantaneously to changes in prices, instead the true demand $d_i$ evolves according to
\begin{equation} \textstyle
	\dot d_i = -\frac{1}{\tau^M}(d_i - \bar d_i (\sigma,d_i^\text{w})),
\end{equation}
where $\tau_M > 0$ is the time constant of the market.

The overall dynamics of the supply-chain system
 are
\begin{equation*}
\small
\begin{bmatrix}
	\dot s_i\!-\!\bar s_i\\
	\dot l_i\\
	\dot d_i
	\end{bmatrix} = \begin{bmatrix}
		0 & 1 & -1\\
		-\frac{k_i}{\tau_i^P} & -\frac{1}{\tau_i^P} & \frac{1}{\tau_i^P}\\
		0 & 0 & -\frac{1}{\tau^M}
	\end{bmatrix} \begin{bmatrix}
		s_i\!-\!\bar s_i\\l_i\\d_i
	\end{bmatrix} + \begin{bmatrix}
		0 \\ 0 \\ \frac{\bar d_i(\sigma,d_i^\text{w})}{\tau^M}
	\end{bmatrix},
\end{equation*}
which is a collection of linear time invariant (LTI) systems of the form $\dot x_i = A_ix_i + B_iu + Dw_i$, $y_i = C_i x_i$, with state $x_i = (s_i\!-\!\bar s_i, \ l_i, \ d_i)$, $y_i = (l_i, \ d_i)$, $u = (\sigma_i,\sigma_{-i})$, and $w_i = d_i^{\text{w}}$. These LTI systems are stable if the control gains in \eqref{eq:lCP} satisfy $k_i \geq 0$. The system-level dynamics are of the same form in \eqref{eq:ctime-dynamics}, where $f$ collects the LTI dynamics and the outputs $y_i$ are only locally available. Thus, Assumption~\ref{ass:system_properties} is satisfied with steady-state maps
$p_i(\sigma,d_i^\text{w}) = (0, \ \bar d_i(\sigma,d_i^\text{w}), \ \bar d_i(\sigma,d_i^\text{w}))$
and $h_i(\sigma,d_i^\text{w}) = (\bar d_i(\sigma,d_i^\text{w}), \bar d_i(\sigma,d_i^\text{w}))$, for all $i \in \mc I$. 
\begin{figure}[t]
	\centering	\includegraphics[width=0.95\columnwidth]{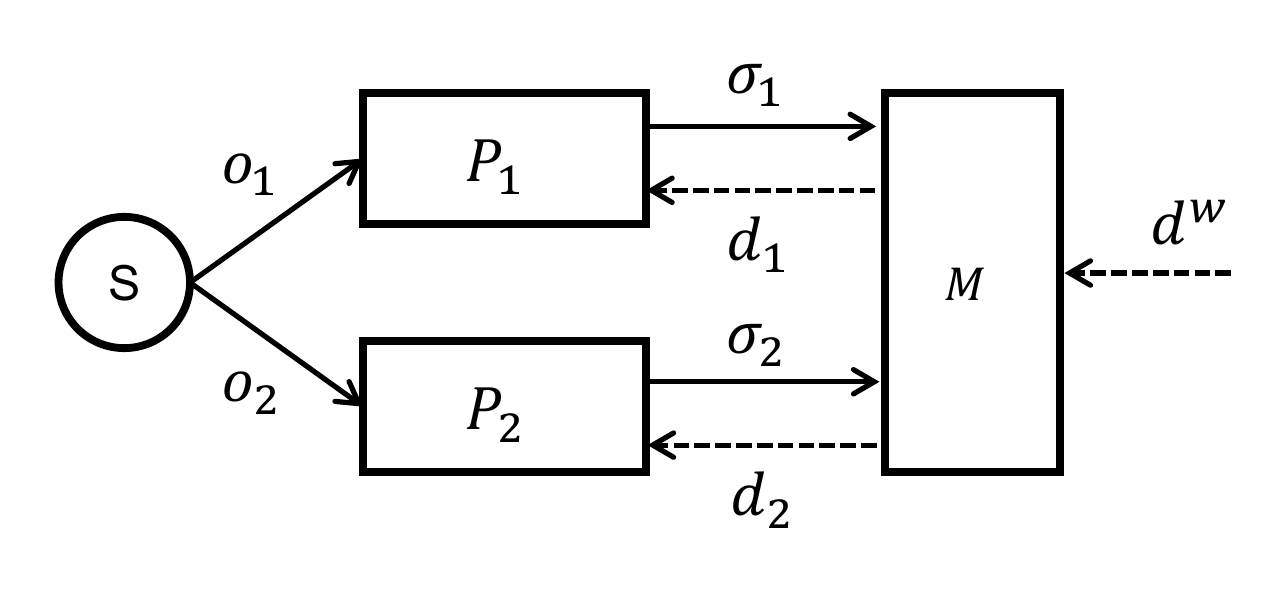}
	\caption{\blue{A supply-chain network with two producers, $P_1$ and $P_2$, sharing a supplier $S$ and a market $M$. Each producers $i$ places an order $o_i$ for raw material form $S$, which is then used to meet the local demand $d_i$. The market $M$ is a dynamical system whose output, namely, the demands $d_i$, is influenced by the producers' selected prices $\sigma_i$ and by other exogenous factors $d^{\text w}$.}}
	\label{fig:supply_chain2}
\end{figure}
 \begin{figure}[t]
  \centering
          \includegraphics[width=\columnwidth]{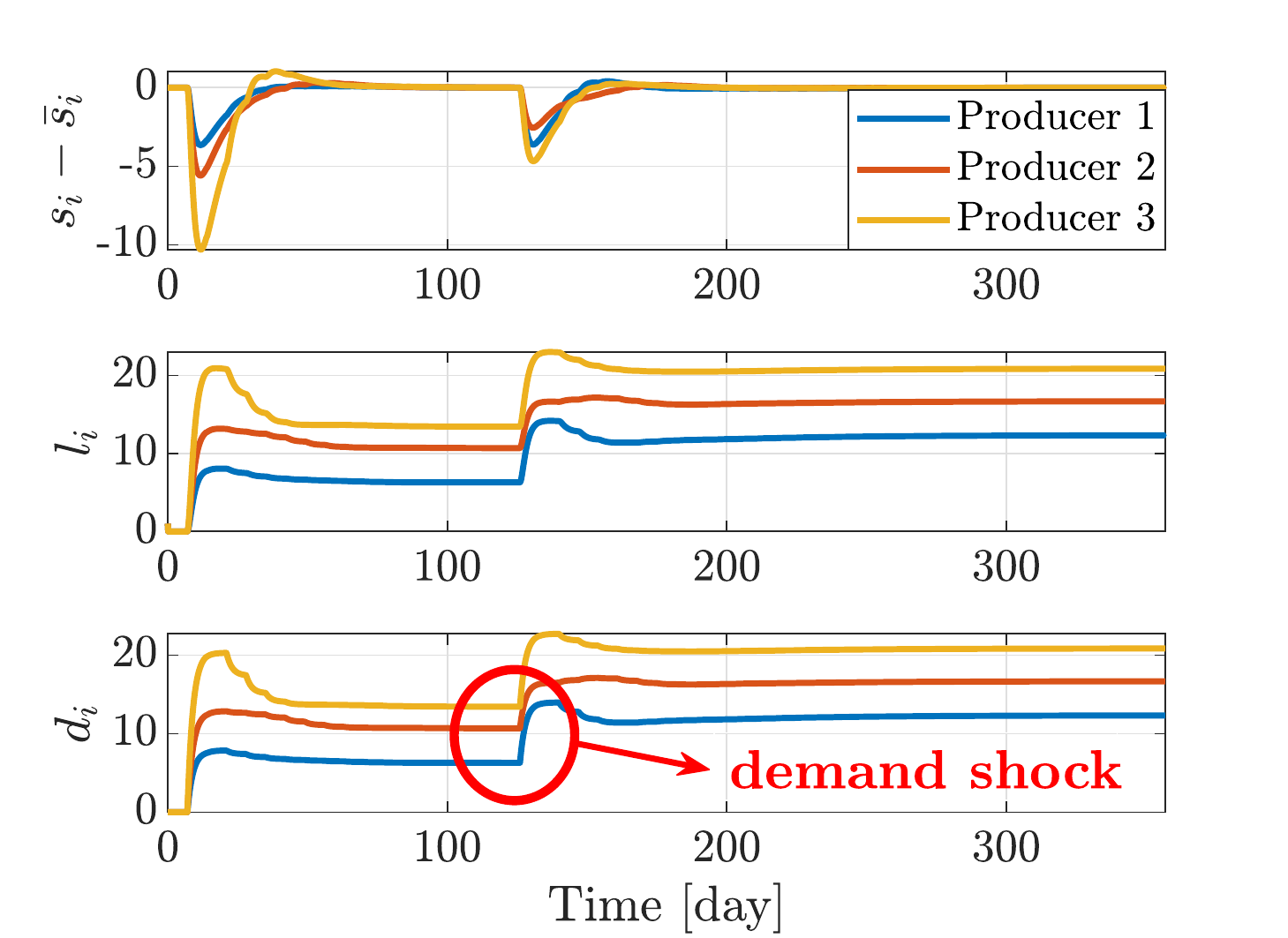}
          \caption{Evolution of the producers' states, i.e., inventory setpoint error ($s_i\!-\!\bar s_i$), production rate ($l_i$), and local demand ($d_i$). The sudden surge of the local demands $d_i$ is the result of a spike of the base-line demand $d_i^{\text w}$.\label{fig:SC_states}}
\end{figure}
\begin{figure}[t]
    \includegraphics[width=\columnwidth]{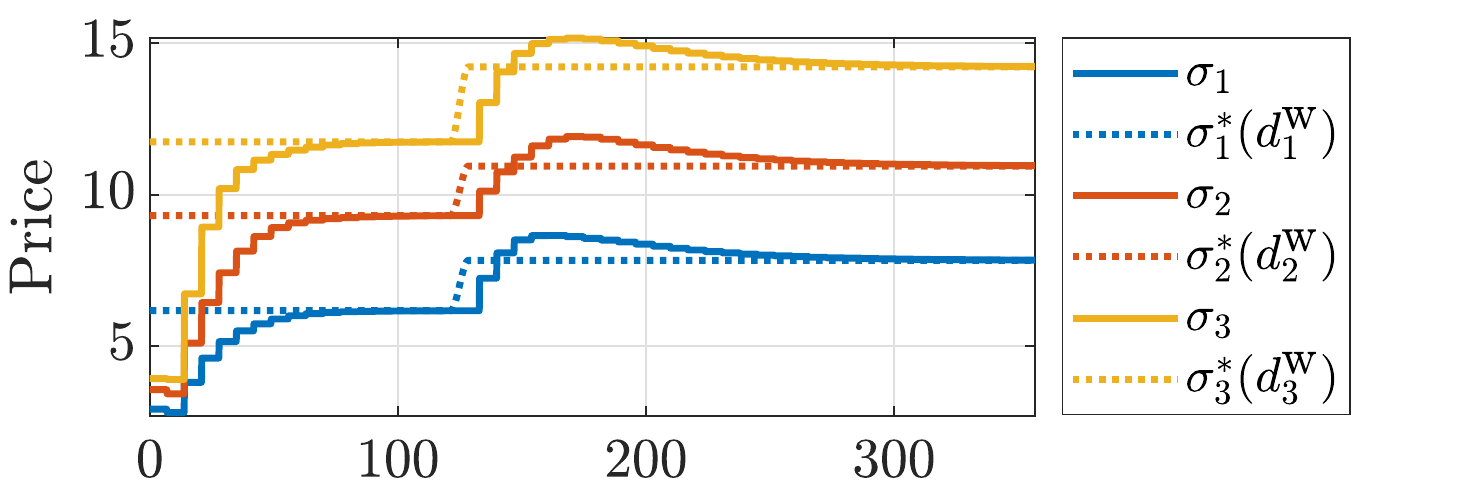}
            \includegraphics[width=\columnwidth]{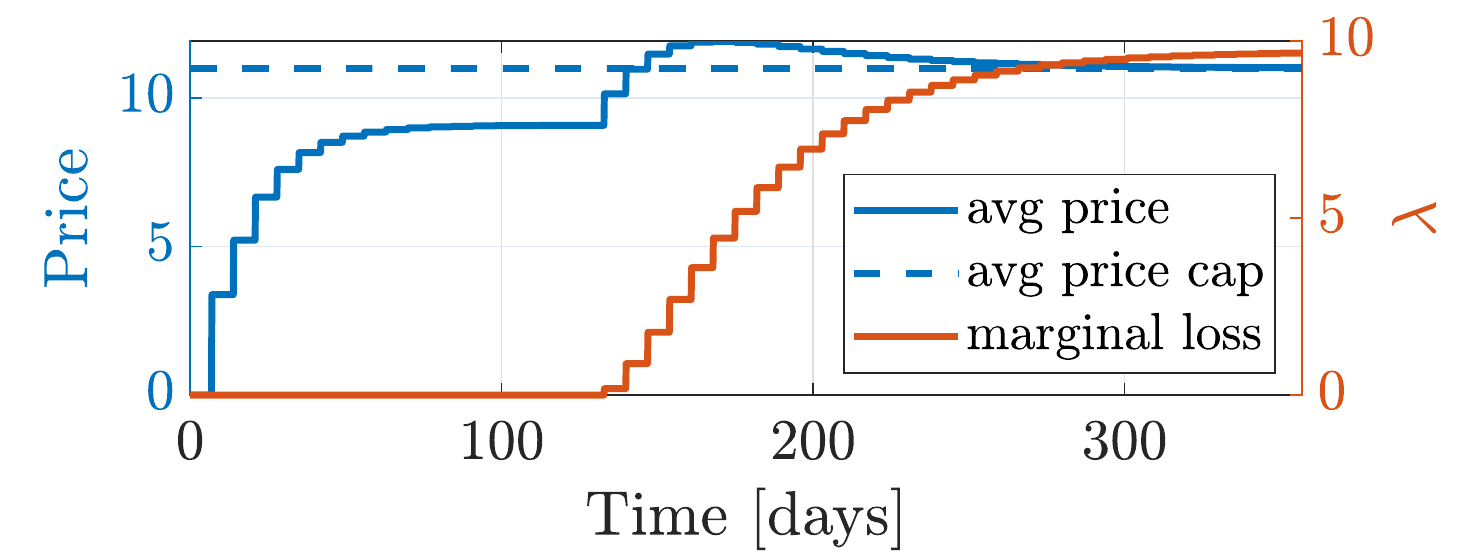}
  \caption{(Top) The prices under the FBS controller (solid lines) track the solution trajectories of the pricing game \eqref{eq:ProdOPT} (dotted line). (Bottom) When the average price exceeds its correspondent cap, the market regulator imposes an additional marginal fee or loss ($\lambda$) to the producers to enforce the constraint.}\label{fig:GNEprices}
\end{figure}

Producers continuously balance local supply and demand using the inventory control law \eqref{eq:lCP} but must choose a price $\sigma_i$ for their product. 
The goal of each producer $P_i$ is to set its prices $\sigma_i$ so local profit is maximized at steady-state operation while the price regulations are respected, i.e.,
\begin{subequations}
\label{eq:ProdOPT}
\begin{align} 
\label{eq:ProdOPTcost}
 \min_{\sigma_i \geq \sigma_i^{\min}} &  \quad
 \underbrace{c_i^\top \bar d_i (\sigma,d_i^\text{w})}_{\text{production cost}}  -\underbrace{\sigma_i^\top  \bar{d}_i(\sigma,d_i^\text{w})}_{\text{sales revenue}} \\[.3em]
  \text{s.t.} 
& \textstyle
\quad \frac{1}{N} \sum_{j \in \mc I} \sigma_j \leq \sigma^{\text{max}}_{\text{avg}}
\label{eq:ProdOPTcnstr2}
\end{align}
\end{subequations}
where $c_i \geq 0$ in \eqref{eq:ProdOPTcost} is the production price for producer $P_i$. The constraints enforce limits on the individuals and average prices, respectively, usually imposed by market regulators (e.g., customers associations). The overall behaviour of the supply chain is coupled through the price/demand curve \eqref{eq:price-curve} as well as the price-cap regulations \eqref{eq:ProdOPTcnstr2}.

In practice, producers do not know the market model \eqref{eq:price-curve} but can only observe its transient outcome, namely, the local demand $d_i$. 
%
Thus, to guide the supply chain to a ``fair" competitive equilibrium (a variational GNE) of the game \eqref{eq:ProdOPT}, we use the FBS controller (Figure~\ref{fig:FBScontr}) developed in Section~\ref{sec:control-strat}.


We consider a pricing game played by $N \!= \!3$ producers, on a single-product market ($\sigma_i \in \mathbb R$), over a year. All the parameters of the supply-chain model and the game are drawn from uniform distributions and fixed over the course of the simulations. The only exception are the baseline demands $d_i^{\text w}$ in \eqref{eq:price-curve}; these suddenly increased by a factor of $3$ at day $120$ to simulate a massive (non-price-related) surge in demand due to, for example, a natural disaster, a cultural event, or the introduction of a new product.

At each sampling period $\tau\!=\!7$ [days], producers measure their local demand $d_i$, receive a marginal fee, or loss, $\lambda$ from the market regulator, and, then, update their prices $\sigma_i$ according to the FBS controller in Figure~\ref{fig:FBScontr}.

We checked numerically that conditions (C1)--(C3) in Section~\ref{sec:control-strat} hold; thus, we can invoke (i) Lemma~\ref{lem:gameReg} to prove that Assumption \ref{ass:strong_reg} is satisfied for the game~\eqref{eq:ProdOPT}, and (ii) Lemma~\ref{lem:FBS} to prove that Assumption \ref{ass:algo} holds for the FBS controller.
In turn, Theorem~\ref{thm:main_theorem} guarantees LISpS of the controlled supply chain system  with respect to exogenous demand fluctuations, under an appropriate choice of the sampling time $\tau$. 

In Figure \ref{fig:SC_states}, we illustrate the evolution of the inventory levels, production rates, and local demands. The local controllers \eqref{eq:lCP} keep the inventories at the desired level despite the initial transient and the sudden demand surge. Moreover, the FBS controller guides the prices to track the variational GNE trajectories of the pricing game \eqref{eq:ProdOPT}, as illustrated in Figure~\ref{fig:GNEprices}~(top). Finally, we note that the sudden spike in demand causes the solution trajectories of the game~\eqref{eq:ProdOPT} to hit the price capping constraints \eqref{eq:ProdOPTcnstr2}. This can be seen in Figure~\ref{fig:GNEprices}~(bottom), in which the average price trajectory exceeds its limit and, in turn, the dual variable (internal state) $\lambda$ of the FBS controller increases to penalize this violation.
Code for this application example is available in~\cite{gitlab_repo}.

\section{Conclusions} \label{ss:conclusions}
Iterative algorithms for solving generalized equations, such as Josephy--Newton, forward-backward splitting, can be used as sampled-data robust feedback controllers for guiding complex unknown dynamical systems to
constrained and economic equilibria.
Under robust stability of the plant, strong regularity of the generalized equation describing the control objective, and robust convergence of the iterative algorithm, the sampled-data algorithm-plant cyber-physical
interconnection is locally input-to-state stable with respect to unmeasured disturbances affecting the plant, provided that the sampling period is appropriately designed.
Illustrative numerical examples in building energy management and supply chain coordination corroborated these theoretical findings.
Future research directions include  incorporating online model (i.e., input-output sensitivities) learning  and real-time constraint satisfaction.

\appendix
\subsection{Proof of Lemma~\ref{lem:gameReg}}
\label{ap:Lemma_gameReg}
Define the auxiliary GE 
\begin{equation} \label{eq:aux_GE}
	\begin{bmatrix}
		F(u,s)\\\tilde b
	\end{bmatrix} + \begin{bmatrix}
0 & \tilde A^\top \\ -\tilde A & 0
\end{bmatrix} \begin{bmatrix} u \\ \lambda
\end{bmatrix} + \begin{bmatrix}
	0 \\ \mc{N}_{\reals^{\tilde m}_{\geq0}} (\lambda)
\end{bmatrix},
\end{equation} where $F$ is defined in \eqref{eq:GNEge}, $\tilde A$ and $\tilde b$ are defined in condition (C3) and $\tilde b$ has $\tilde m$ rows, that results from dualizing all constraints in the game \eqref{eq:Game}. The GE \eqref{eq:aux_GE} is strongly regular if the LICQ (C3), $F$ is Lipschitz w.r.t. $s$, and the strong second-order sufficient condition (SSOSC) \cite[Eq. (1.49)]{izmailov2014newton} holds by \cite[Prop 1.27,1.28]{izmailov2014newton}. Strong monotonicity of $F$ (C1), linearity of the constraints is sufficient for the SSOSC and thus for strong regularity of \eqref{eq:aux_GE} \cite[Prop 1.27,1.28]{izmailov2014newton}. Since the original GE \eqref{eq:GNEge} dualizes a subset of the constraints in \eqref{eq:aux_GE} the solution map of \eqref{eq:GNEge} can be constructed by selecting a subset of the dual variables in \eqref{eq:aux_GE} and is thus also strongly regular.
{\hfill $\blacksquare$}

\subsection{Proof of Lemma~\ref{lmm:JN_method}}
\label{app:JN_method}
By \cite[Theorem 10]{liao2020time} under these assumptions there exists $0 < \tilde \epsilon \leq \bar \epsilon$ such that if $|z - \bar z(w)| \leq \tilde \epsilon$ then there exists $\tilde \eta \in (0,1)$ such that $|\mathbb{T}(z,w) - \bar z| \leq \tilde \eta|z - \bar z|$ which immediately implies Assumption~\ref{ass:algo} (i) and (ii). Assumption~\ref{ass:algo} (iii) follows from the Lipschitz continuity of $(H(z) (\cdot) + \mc{A}(\cdot))^{-1}$ and the fact that $\mathbb{G}$ is $1$-Lipschitz continuous with respect to $s$. 
{\hfill $\blacksquare$}

\subsection{Proof of Proposition~\ref{lmm:SQNE}}
\label{ap:SQNE}
To prove \eqref{eq:sQNE-merit-2}, let us fix $w \in \mc W$ and
define $\zeta(z,w) = |\mathbb{T}(z,w) - z|^2$, we see that (i) $\zeta(z,w) = 0$ if and only if $z \in \mathrm{fix}~\mathbb{T}(\cdot,w)$, (ii) since $\tilde  T(\cdot,w)$ is continuous, by assumption, so is $\zeta(\cdot,w)$, and (iii) $\zeta(\cdot,w)$ is positive definite. Thus, by invoking \cite[Lemma 4.3]{khalil2002nonlinear}, it follows that there exists a $ \K$-function $ \alpha_w$ such that
\begin{equation} \label{eq:SQNE_Khalil}
	  \alpha_w(|z- \bar z|) \leq \zeta(z,w), \quad \bar z \in \fix~\mathbb{T}(\cdot,w).
\end{equation}
Now, since $\mathbb{T}(\cdot,w)$ is $\rho$-SQNE (Definition \ref{def:SQNE}), we have that
\begin{subequations}
\begin{align}
	{|\tilde  T(z,w) - \bar z|}_P^2 & \leq {|z - \bar z|}_P^2 - \rho  {|\tilde  T(\cdot,w)(z) - z|}_P^2\\
	&\leq {|z - \bar z|}_P^2 - \rho \lambda_{\min}(P)   \alpha_w(|z- \bar z|),
	\label{eq:sineq}
\end{align}
\end{subequations}
for some $P \succ 0$, where for the second inequality \eqref{eq:sineq} we used ${|\cdot|}_P \geq \lambda_{\min}(P)|\cdot| $ and \eqref{eq:SQNE_Khalil}, sequentially.
Since $\mathcal W$ is compact by assumption, we can define the $\mc K$-function $\underline \alpha(z) = \min_{w \in \mc W}   \alpha_w (z)$ \cite[Lemma~1]{picallo2021cross}. Finally, letting $\alpha(\cdot) = \rho \lambda_{\min}(P) \underline \alpha(\cdot) \in \K$ completes the proof.
{\hfill $\blacksquare$}

\subsection{Proof of Lemma~\ref{lem:FBS}}
\label{ap:FBS}
We show that conditions (i)-(iii) of Assumption~\ref{ass:algo} hold.\\
(i) It follows by \cite[Prop.~26.1~(iv)~(a)]{bauschke2011convex} that $\fix~\mathbb{T}(\cdot,w) = \zer(\Phi^{-1}\mathbb{G}(\cdot,w)+\Phi^{-1} \mc A) = \zer(\mathbb{G}(\cdot,w)+ \mc A)= S(w)$, for all $w \in \mc W$, where $S(w)$ is a singleton since the pseudo-gradient $\tilde{F}(\cdot,w)$ is strongly monotone by (C1).\\
%
(ii)
It follows by \cite[Prop.~26.1~(iv)~(d)]{bauschke2011convex} that $\mathbb{T}(\cdot,w)$ is $\eta$-averaged with $\eta=2\beta\delta/(4\beta \delta-1)$, for all $w \in \mc W$, since $\Phi^{-1}\mathbb{G}(\cdot,w)$ is $\beta\delta$-cocoercive, with $\beta = 2\tilde \mu/\tilde \ell^{\,2}$, $\Phi^{-1}\mc A$ is maximally monotone, both with respect to the norm $\|\cdot\|_\Phi$, see \cite[\S~III.B]{belgioioso2021semi}.
In turn, every $\eta$-averaged operator is continuous, since nonexpansive, and $\rho$-SQNE, with $\rho=\frac{1-\eta}{\eta}$, by \cite[Prop.~4.35]{bauschke2011convex}. Moreover, $\fix~T(\cdot,w)=S(w)$ is a singleton by part (i) of this proof.
Hence, we can invoke Proposition~\ref{lmm:SQNE} to prove that Assumption~\ref{ass:algo}~(ii) is satisfied, with $W(z,w)={|z-S(w)|}_\Phi^2$, $\alpha_1 = \lambda_{\min}(\Phi)$, $\alpha_2 = \lambda_{\max}(\Phi)$, and $\alpha = \rho \lambda_{\min}(\Phi) \underline \alpha $, for some $\underline \alpha \in \mc K$ and $\Phi$ as in \eqref{eq:Phi}.\\
(iii) By \cite[Prop.~23.8]{bauschke2011convex}, the resolvent ${(\id+ \Phi^{-1} \mc A)}^{-1}$ in \eqref{eq:pFB} is (firmly-)nonexapansive, since $\Phi^{-1} \mc A$ is maximally monotone. By combining this result with the $\ell$-Lipschitz continuity of $F$, i.e., (C2), we can show that Assumption~\ref{ass:algo}~(iii) is satisfied with $L_{T}:= \ell/\lambda_{\min}(\Phi)$.
{\hfill $\blacksquare$}

\subsection{Explicit expressions for the functions in \eqref{eq:discrete-time-error-system}}%
\label{app:expressions}%
\vspace*{-1.5em}
\begin{align*}
	&\mc{G}_1(x^k,u^k,\Delta u^k,\tilde w^k) = \\
 	&\hspace*{10em} \hspace*{-3em} \psi(t^k,x^k,u^k, \tilde w^k) - p(u^k+\Delta u^{k},w^{k+1})\\
&\mc{G}_2(e^k,\delta x^k,\tilde w^k) = \hat{T}(e^k,\delta x^k,\tilde w^k) - \bar z(w^{k+1})
\\
&	\mc{H}(e^k,\delta x^k,\tilde w^k) = q(\hat{T}(e^k,\delta x^k,\tilde w^k))  - q(e^k + \bar z(w^k))\\
&
\hat{T}(e^k,\delta x^k, \tilde w^k) = T(e^k + \bar z(w^k), \\ &\hspace*{3.5em} g(\psi(\delta x^k + p(q(z^k),w^k) , q(e^k + \bar z^k),\tilde w^k),w^{k+1})
\end{align*}
%
%

\subsection{Proof of Theorem~\ref{thm:plant_LISS}}
\label{ap:plant_LISS}
We begin with a preparatory Lemma.
\begin{lmm} \label{lmm:V_decrease}
Let Assumption~\ref{ass:system_properties} hold. Then
\begin{equation}
	V(x^{k+1},v ,w^{k+1}) \leq e^{-\alpha_5 \tau} V(x^k,v ,w^k) + \sigma_1(d^k),
\end{equation}
where $x^{k+1} = \psi(t^k,x^k,v,w)$.
\end{lmm}
\begin{proof}
For any $t \in [t^k,~~t^{k+1}]$, Assumption~\ref{ass:system_properties} and the comparison principle \cite[Lemma 3.4]{khalil2002nonlinear} imply that
\begin{align}
  &V(x(t),v,w(t)) \leq e^{-\alpha_5 (t-t^k)} V^k + \int_{t^k}^t e^{(s-t)} \sigma_1(|\dot w(s)|)~ds \nonumber \\
  &\leq e^{-\alpha_5 (t-t^k)} V^k + \sigma_1\left(\esssup_{s\in T} |\dot w(s)|\right) \int_{t^k}^t e^{-(t-s)}~ds\nonumber\\
 & = e^{-\alpha_5 (t-t^k)} V^k + (1-e^{-\alpha_5 (t-t^k)})\sigma_1(d) \nonumber
\\
  &\leq \blue{e^{-\alpha_5 (t-t^k)} V^k +  \sigma_1(d)}, \hspace*{6em} \label{eq:boundedness}
\end{align}
where $V^k \!= \! V(x^k,v,w^k)$. Let $t \!=\! t^{k+1}$ and conclude.
\end{proof}

With this result in hand, we proceed to show LISS of the plant. Let $V^k = V(x^k,u^k,w^k)$, using Lemma~\ref{lmm:V_decrease} with $v = u^{k+1}$ implies that
\begin{align}
	&V^{k+1} \leq e^{-\alpha_5\tau} V(x^k,u^{k+1},w^k) + \sigma_1(d^k)\nonumber\\
	&\leq e^{-\alpha_5 \tau} V^k + e^{-\alpha_5 \tau}|V(x^k,u^{k+1},w^k)-V^k| + \sigma_1(d^k) \nonumber\\
	&\leq e^{-\alpha_5 \tau} V^k+ e^{-\alpha_5 \tau}L_V |u^{k+1} - u^k| + \sigma_1(d^k) \label{eq:pf5}
\end{align}
where $L_V>0$ is the Lipschitz constant of $V$ (which exists since $V$ is continuously differentiable and $\mc{U}$ is compact). Continuing, we note that $e^{-\alpha_5 \tau} \in (0,1)$ for any $\tau > 0$ and thus by the comparison principle, see e.g, \cite[Example 3.4]{jiang2001input},
\begin{align} 
V^k &\leq (e^{-\alpha_5\tau})^k V^0 + \frac{e^{-\alpha_5\tau} L_V}{1-e^{-\alpha_5\tau}} \|\Delta u\| + \frac{e^{-\alpha_5\tau}}{1-e^{-\alpha_5\tau}} \sigma_1(\|d\|) \nonumber\\
& \leq e^{-k\alpha_5\tau} V^0 + \gamma_x^u(\tau) \|\Delta u\| + \gamma_x^d(\|d\|,\tau)\label{eq:pf7}
\end{align}
where $\gamma_x^u(\tau) = \frac{e^{-\alpha_5\tau} L_V}{1-e^{-\alpha_5\tau}}$ and $\gamma_x^d(s,\tau)= \frac{e^{-\alpha_5\tau}}{1-e^{-\alpha_5\tau}} \sigma_1(s)$.

Assumption~\ref{ass:system_properties} requires that $V(x(t),u(t),w(t)) \leq \epsilon_x$. We first show this holds at the sampling instants, i.e., $V(x^k,u^k,w^k) \leq \epsilon_x$ then demonstrate inter-sample satisfaction under the conditions
\begin{gather}
	V^0 \leq 0.25 \eps_x,~~\|u\| \leq 0.125 \gamma_x^u(\tau)^{-1} \eps_x,~~\|d\|\leq \eps_w \nonumber\\
	\gamma_x^d(\|d\|,\tau) \leq 0.125 \eps_x \text{ and } \sigma_1(\|d\|)\leq 0.5 \eps_x \label{eq:Vres}
\end{gather}

We proceed by induction to show that $V^k \leq 0.5 \eps_x$. For $k = 0$ its clear that $V^0 \leq 0.25 \eps_x \implies V^0 \leq 0.5\eps_x$. Next assume \eqref{eq:pf5} holds up to $k - 1$, it follows that
\begin{align*}
	V^k &\leq e^{-k\alpha_5\tau} V^0 + \gamma_x^u(\tau) \|\Delta u\| + \gamma_x^d(d,\tau)\\
	&\leq 0.25  \eps_x + \gamma_x^u(\tau)  \frac{0.125}{\gamma_x^u(\tau)} \eps_x  + 0.125 \eps_x =  0.5\eps_x
\end{align*}
as required. Between sampling instants by Lemma~\ref{lmm:V_decrease} we have that for all $t \in [t^k,t^{k+1}]$
\begin{align*}
V(x(t),u^k,w(t)) &\leq e^{-\alpha_5 t} V^k + \sigma_1(d^k) \\
& \leq \max\{2V^k,2\sigma_1(d^k)\}\\ &\leq \max\{2\cdot 0.5 \epsilon_x,2 \cdot 0.5 \epsilon_x\} = \epsilon_x,
\end{align*}
using \eqref{eq:Vres}. Thus we can define $\pi_1 = 0.125 \gamma_x^u(\tau)^{-1}$, and $\pi_2(s) = \max\{s,8\gamma_x^d(s,0),2\sigma_1(s)\}$. Letting $\beta_x(s,r) = e^{-\alpha_5 r} s$ completes the proof (see Definition~\ref{def:LISS}). {\hfill $\blacksquare$}





\subsection{Proof of Theorem~\ref{thm:algo_LIOS}}
\label{ap:algo_LIOS}
We begin with two preparatory lemmas that are used to bound the disturbances and the distance between steps on and off the steady-state manifold. The first is an immediate consequence the bound on $\|\dot w\|$ in Assumption~\ref{ass:system_properties} and the properties of integrals.
\begin{lmm} \label{lmm:Delta_w_bound}
Under Assumption~\ref{ass:system_properties} $|w^{k+1} - w^k| \leq \tau d^k~\forall k \geq 0$.
\end{lmm}

\begin{lmm} \label{lmm:T_Ttilde_bound}
Given Assumptions~\ref{ass:system_properties}-\ref{ass:algo}, there exists $\sigma_2\in \KL$, $\sigma_3\in \K$ such that
\begin{equation}
	|\mathbb{T}(z^k,w^{k+1}) - T(z^k,y^{k+1})| \leq \sigma_2(V^k,\tau) + \sigma_3(d^k).
\end{equation}
where $\mathbb{T}(z,w) = T(z,h(q(z),w))$, provided that $V^k = V(x^k,u^k,w^k) \leq \eps_x$ and $d \leq \eps_w$.
\end{lmm}
\begin{proof}
We adopt the shorthand notation $x^k = x$, $x^{k+1} = x^+$ etc. To begin, note that
\begin{align}
&|\mathbb{T}(z,w^+) - T(z,y^+)| = |T(z,h(q(z),w^+)) - T(z,y^+)| \nonumber\\
& = |T(z,g(p(u,w),w^+)) - T(z,g(x^+,w^+))| \nonumber\\
&\leq L_T |g(x^+,w^+) - g(p(u,w^+),w^+)| \nonumber\\
&\leq L_T L_g |x^+ - p(u,w^+)| \label{eq:pf1}
\end{align}
where we have used $L_T > 0$ and $L_g > 0$ Lipschitz continuity of $T$ (Assumption~\ref{ass:algo}) and of $g$ (Assumption~\ref{ass:system_properties}), respectively. Next we invoke Lemma~\ref{lmm:V_decrease} to show that
\begin{align*}
|x^+ - p(u,w^+)| &\leq \alpha_3^{-1}\circ V(x^+,u,w^+)\\
&\leq \alpha_3^{-1}\left( e^{-\alpha_5\tau} V(x,u,w) + \sigma_1(d) \right)\\
&\leq \alpha_3^{-1}\left( 2e^{-\alpha_5\tau} V(x,u,w)\right)  + \alpha_3^{-1}\left(2\sigma_1(d) \right)
\end{align*}
whenever $V(x,u,w) \leq \eps_x$ and $d\leq \eps_w$. Combining the above inequality with \eqref{eq:pf1} yields
\begin{equation}
	|\mathbb{T}(z,w^+) - T(z,y^+)| \leq \sigma_2(V,\tau) + \sigma_3(|d|)
\end{equation}
as claimed, with $\sigma_2(s,\tau) = L_T L_g\alpha_3^{-1}\left( 2e^{-\alpha_5 \tau}s \right)$ and $\sigma_3 = L_T L_g \alpha_3^{-1} \circ 2 \sigma_1$.
\end{proof}

Now, we are ready to proceed and show LISS of \eqref{eq:sys2e}.

We begin with
\begin{multline} \label{eq:pf12}
W^{k+1} = [W(\mathbb{T}(z^k,w^{k+1}),w^{k+1})]  \\+ [W(z^{k+1},w^{k+1}) - W(\mathbb{T}(z^k,w^{k+1}),w^{k+1})]
\end{multline}
Focusing on the first term, by virtue of Assumption~\ref{ass:algo} we have, if the restriction $|z^k -\bar z^{k+1}| \leq \eps_z$ holds, that
\begin{align*}
W(\mathbb{T}(z^k,w^{k+1}),w^{k+1}) \leq W(z^k,w^{k+1}) - \alpha(|z^k - \bar z^{k+1}),
\end{align*}
where $\bar z^{k+1} = \bar z(w^{k+1})$. Now consider
\begin{align}
\alpha\left(\frac12|z^k - \bar z^k|\right) &= \alpha\left(\frac12|z^k - \bar z^{k+1} + \bar z^{k+1} - \bar z^k|\right) \nonumber\\
&\leq \alpha(|z^k - \bar z^{k+1}|) + \alpha(L_z/2 |w^{k+1} - w^k|),\nonumber
\end{align}
where we have used the weak triangle inequality and Lipschitz continuity of $\bar z$. Rearranging the last line yields that
\begin{equation}
	\alpha(|z^k - \bar z^{k+1}|) \geq \tilde\alpha(|e^k|) - \tilde \alpha(L_z|\Delta w^k|),
\end{equation}
where $\tilde \alpha = \alpha \circ 0.5 \in \Kinf$. Therefore
\begin{align}
	&W(\mathbb{T}(z^k,w^{k+1}),w^{k+1}) \leq W(z^k,w^{k+1}) - \alpha(|z^k - \bar z^{k+1}|) \nonumber\\ 
    &\leq W(z^k,w^{k+1}) - \tilde \alpha(|e^k|) + \tilde \alpha(L_z |\Delta w|^k) \nonumber\\
    & \leq W^k + |W(z^k,w^{k+1}) -W^k|  - \tilde \alpha(|e^k|) + \tilde \alpha(L_z |\Delta w|^k) \nonumber \\
    & \leq W^k - \tilde\alpha(|e^k|) + (\omega_1 + \tilde \alpha \circ L_z~\id)(|\Delta w^k|) \nonumber\\
    & \leq W^k - \tilde\alpha(|e^k|) + ((\omega_1 + \tilde \alpha \circ L_z~\id) \circ \tau~\id)(d^k) \label{eq:pf11}
\end{align}
where $\omega_1\in \K$ exists since $W$ is uniformly continuous in $w$ and the last line uses Lemma~\ref{lmm:Delta_w_bound}. Uniform continuity holds by the Heine–Cantor theorem as $W$ is continuous and $\{z~|~|z-\bar z(w)|~~\forall w\in \mc{W}\} \times \mc{W}$ is compact.

Next, we focus on the second term in \eqref{eq:pf12}, we see that
\begin{align*}
\Delta W & = |W(z^{k+1},w^{k+1}) - W(\mathbb{T}(z^k,w^{k+1}),w^{k+1})|\\ 
& \leq \omega_1(|z^{k+1} -\mathbb{T}(z^k,w^{k+1})|)\\
& = \omega_1(|T(z^k,y^{k+1}) -\mathbb{T}(z^k,w^{k+1})|).
\end{align*}
Continuing, we invoke Lemma~\ref{lmm:T_Ttilde_bound} to show that
\begin{align}
\Delta W &\leq \omega_1(\sigma_2(V^k,\tau) + \sigma_3(d^k)) \nonumber\\
&\leq \omega_1(2\sigma_2(V^k,\tau)) + \omega_1(2\sigma_3(d^k)) \nonumber\\
&\leq \omega_2(V^k,\tau) + \omega_3(d^k), \label{eq:pf13}
\end{align}
where $\omega_2 = \omega_1 \circ 2 \sigma_2 \in \KL$ and $\omega_3 = \omega_1\circ 2\sigma_3 \in \K$, subject to the restrictions $V^k\leq \eps_x$ and $d^k \leq \eps_w$.

Combining \eqref{eq:pf11} and \eqref{eq:pf13} to bound \eqref{eq:pf12} we obtain that
\begin{multline*}
	W^{k+1} - W^k \leq - \tilde\alpha(|e^k|) +  ((\omega_1 + \tilde \alpha \circ L_z~\id) \circ \tau~\id)(d^k)\\+ \omega_2(V^k,\tau) + \omega_3(d^k).
\end{multline*}
Defining $\omega_4 = (\omega_1 \circ \tau  + \tilde \alpha \circ L_z \tau  + \omega_3) \in \K$ we obtain
\begin{equation} \label{eq:pf14}
	W^{k+1} - W^k \leq - \tilde\alpha(|e^k|) + \omega_2(V^k,\tau) + \omega_4(d^k).
\end{equation}

Next we focus on the restrictions required for \eqref{eq:pf14} to hold, these are $V^k \leq \eps_x$, $d^k\leq \eps_w$ and $|z^k - \bar z^{k+1}| \leq \eps_z$. The first two follow from the assumptions. The third expression can be bounded as follows
\begin{align*}
|z^k - \bar z^{k+1}| &= |z^k - \bar z^{k} + \bar z^k - \bar z^{k+1}|\\
&\leq |z^k - \bar z^{k}| + |\bar z^k - \bar z^{k+1}| \leq |e^k| + L_z |\Delta w^k| \\
& \leq |e^k| + L_z\tau d^k \leq \max\{2 |e^k|,2L_z\tau d^k\},
\end{align*}
and thus a sufficient condition for $|z^k - \bar z^{k+1}| \leq \eps_z$ is $|e^k| \leq 0.5\eps_z$ and $d^k \leq \frac{0.5 \eps_z}{\tau L_z}$.

Thus the dissipation inequality \eqref{eq:pf14} holds in a neighbourhood of the origin given by $|e^k|\leq 0.5 \eps_z$, $d^k \leq \min\{\frac{0.5 \eps_z}{\tau L_z},\eps_w\}$, and $V^k \leq \eps_x$, and the system is locally asymptotically stable with zero inputs (0-LAS) and thus LISS\cite{jiang2004nonlinear}, i.e., there exists $\beta_z,\gamma_z^x,\gamma_z^d, \veps_1,\veps_2$ and $\veps_3$ such that
\begin{equation} \label{eq:lios-3}
	W^k\leq \beta_z(W^0,k) + \gamma_z^x(\|V\|,\tau) + \gamma_z^d(\|d\|)
\end{equation}
for $W^0 \leq \veps_1$, $\|V\| \leq \veps_2$, and $\|d\| \leq \veps_3$, proving \eqref{eq:algo-lios1}. That $\gamma_z^x$ can be chosen such that $\gamma_z^x\in \KL$ follows from \cite[Lemma 3.13]{jiang2001input} and $\omega_2\in \KL$. Next, we bound $\Delta u$.

\begin{lmm} \label{lmm:Delta_u_bound}
Let Assumptions~\ref{ass:system_properties} - \ref{ass:algo} hold and define $\Delta u^k = u^{k+1} - u^k$. Then $\exists$ $\sigma_4\in \KL$ and $\sigma_5,\sigma_6 \in \K$ such that
\begin{equation}
|\Delta u^k | \leq \sigma_4(V^k,\tau) + \sigma_5(W^k) + \sigma_6(d^k),
\end{equation}
whenever $V^k = V(x^k,u^k,w^k) \leq \eps_x$, $W^k = W(z^k,w^k) \leq \alpha_2(\frac12\eps_z)$, and $d^k\leq \min(\eps_w,\frac{0.5 \eps_z}{\tau L_z})$ where $L_z$ is the Lipschitz constant of $\bar z$.
\end{lmm}
\begin{proof}
We adopt the shorthand notation $z^k = z$, $z^{k+1} = z^+$ and so on throughout. We begin by noting that
\begin{align}
|\Delta u| &= |q(T(z,y^+)) - q(z)|\nonumber\\
&\leq L_q |T(z,y^+) - z|\nonumber\\
& = L_q |\mathbb{T}(z,w^+) - z + T(z,y^+) -\mathbb{T}(z,w^+)|\nonumber\\
& \leq L_q|\mathbb{T}(z,w^+) - z| + L_q |T(z,y^+) -\mathbb{T}(z,w^+)| \label{eq:pf2}
\end{align}
where we have added and subtracted $\mathbb{T}$ and used that $q$ is $L_q$-Lipschitz (Assumption~\ref{ass:strong_reg}). The first term in \eqref{eq:pf2} can be bounded as follows,
\begin{align*}
& |\mathbb{T}(z,w^+) - z| = |\mathbb{T}(z,w^+) - \bar z(w^+) - z + \bar z(w^+)|\\
& \leq |\mathbb{T}(z,w^+) - \bar z(w^+)| + |z - \bar z(w^+)|\\
& \leq \alpha_1^{-1}\left(W(\mathbb{T}(z,w^+),w^+)\right) + |z - \bar z(w^+)|
\end{align*}
where we have used \eqref{eq:Wbounds} from Assumption~\ref{ass:algo}. Next we note that since $T$ is converging (Assumption~\ref{ass:algo}) we can use \eqref{eq:Wdot} to conclude that $W(\mathbb{T}(z,w^+),w^+) \leq W(z,w^+)$ and thus whenever $|z-\bar z(w^+)| \leq \eps_z$ (to be shown below)
\begin{align}
|\mathbb{T}(z,w^+) - z| &\leq \alpha_1^{-1}\left(W(\mathbb{T}(z,w^+),w^+)\right) + |z - \bar z(w^+)| \nonumber\\
& \leq \alpha_1^{-1}\left(W(z, w^+)\right) + |z - \bar z(w^+)| \nonumber\\
& \leq \alpha_1^{-1}\left(W(z, w^+)\right) + \alpha_1^{-1}(W(z,w^+)) \nonumber\\
& = 2\alpha_1^{-1} (W(z,w^+)) \nonumber\\
& = 2\alpha_1^{-1} (W(z,w) + W(z,w^+) - W(z,w)) \nonumber\\
& \leq 2\alpha_1^{-1} \left(W(z,w) + \omega_1(|w^+ - w|)\right)\nonumber\\
& \leq 2\alpha_1^{-1} (2W(z,w)) + 2\alpha_1^{-1} \circ 2 \omega_1(|w^+ - w|)\nonumber\\
& = \rho_1(W(z,w)) + \rho_2(|\Delta w|) \label{eq:pf3}
\end{align}
where $\rho_1 = 2\alpha_1^{-1} \circ 2$, $\rho_2 = 2\alpha_1^{-1} \circ 2 \omega_1$, and $\omega_1\in \K$ exists by uniform continuity of $W$ over the compact set $\mc{W}$.

Using \eqref{eq:pf3} and Lemma~\ref{lmm:T_Ttilde_bound} to bound the first and second terms in \eqref{eq:pf2} respectively yields that
\begin{align*}
	|\Delta u| &\leq L_q \rho_1(W) + L_q \rho_2(|\Delta w|) + L_q\sigma_2(V,\tau) + L_q\sigma_3(d).
\end{align*}
Using that $|\Delta w| \leq \tau d$ by Lemma~\ref{lmm:Delta_w_bound} and collecting terms we obtain that
\begin{equation} \label{eq:pf4}
	|\Delta u | \leq \sigma_4(V,\tau) + \sigma_5(W) + \sigma_6(d).
\end{equation}
where $\sigma_4 = L_q \sigma_2$, $\sigma_5 =L_q\rho_1 = 2L_q\alpha_1^{-1} \circ 2$, and $\sigma_6 =L_q (\rho_2 \circ \tau + \sigma_3) = L_q(2\alpha_1^{-1} \circ 2 \omega_1 \circ \tau + \sigma_3)$ as claimed.

Regarding restrictions, \eqref{eq:pf4} holds when $V \leq \eps_x$, $d \leq \eps_w$ and $|z-\bar z(w^+)| \leq \eps_z$. To satisfy the last, we note that
\begin{align*}
|z-\bar z(w^+)| &\leq |z - \bar z(w)| + |\bar z(w^+) - \bar z(w)| \\
&\leq |e| + L_z \tau d \leq \max\{2|e|,2L_z \tau d\}
\end{align*}
and thus a sufficient condition for $|z-\bar z(w^+)|\leq \eps_z$ is $|e| \leq 0.5 \eps_z$, which is implied by $W \leq \alpha_2(0.5 \eps_z)$, and $d\leq \frac{0.5\eps_z}{\tau L_z}$. 
\end{proof}




We are now ready to prove LIOS \eqref{eq:algo-lios2}. Combining \eqref{eq:lios-3} and Lemma~\ref{lmm:Delta_u_bound} we obtain
\begin{multline}
|\Delta u^k| \leq \sigma_5\left( \beta_z(W^0,k) + \gamma_z^x(\|V\|,\tau) + \gamma_z^d(\|d\|)\right) \\
+ \sigma_4(V^k,\tau) + \sigma_6(d^k),
\end{multline}
and, using the weak triangle inequality, 
\begin{multline}
|\Delta u^k| \leq \sigma_5\left( 2\beta_z(W^0,k)\right ) + \sigma_5\left(2\gamma_z^x(\|V\|,\tau)\right) \\ + \sigma_5\left(2\gamma_z^d(\|d\|)\right)
+ \sigma_4(V^k,\tau) + \sigma_6(d^k).
\end{multline}
Finally, collecting terms we see that
\begin{align}
|\Delta u^k| &\leq \sigma_5\left( 2\beta_z(W^0,k)\right ) + (\sigma_4 + \sigma_5 \circ 2\gamma_z^x)(\|V\|,\tau) \nonumber \\ & \qquad \qquad \qquad \quad \quad~+ (\sigma_5 \circ 2 \gamma_z^d + \sigma_6)(\|d\|) \nonumber\\
& = \beta_u(W^0,k) + \gamma_u^x(\|V\|,\tau) + \gamma_u^d(\|d\|)
\end{align}
as claimed, where $\beta_u = \sigma_5 \circ 2 \beta_z$, $\gamma_u^x = \sigma_4 + \sigma_5 \circ 2\gamma_z^x$ and $\gamma_u^d = \sigma_5 \circ 2 \gamma_z^d + \sigma_6$. The restrictions are simply the intersection of the preconditions of \eqref{eq:lios-3} and Lemma~\ref{lmm:Delta_u_bound}, and thus $\bar \veps_1 = \min(0.5\veps_1, \alpha_2(0.5\eps_z))$, $\bar\veps_2 = \min(\eps_x, \veps_x)$, $\bar \veps_3 = \min(\eps_w ,\frac{0.5\eps_z}{\tau L_z}, \veps_3)$.
{\hfill $\blacksquare$}

\subsection{Proof of Theorem~\ref{thm:dt_coupled_iss}}
\label{ap:dt_coupled_iss}

Let $V^k = V(x^k,u^k,w^k)$ and $W^k = W(z^k,w^k)$. Thanks to Theorem~\ref{thm:plant_LISS} and \cite[Lemma 3.8]{jiang2001input} we have that
\begin{equation} \label{eq:thm41}
	\limsupk |V^k| \leq \gamma_x^u(\tau) \limsupk |\Delta u^k| + \gamma_x^d\left(\limsupk |d^k|,\tau\right)
\end{equation}
for $V^0$, $\|\Delta u\|$, and $\|d\|$ sufficiently small. Further, using Theorem~\ref{thm:algo_LIOS} and \cite[Lemma 3.8]{jiang2001input} we have that
\begin{equation} \label{eq:thm42}
 	\limsupk |\Delta u^k| \leq \gamma_u^x\left(\limsupk V^k ,\tau
 	\right)  + \gamma_u^d\left(\limsupk d^k\right),
 \end{equation} 
for $W^0$, $V_0$, and $\|d\|$ sufficiently small. Combining \eqref{eq:thm41} and \eqref{eq:thm42} we obtain that
\begin{equation}
	\limsupk V^k \leq \gamma_x^u(\tau) \gamma_u^x\left(\limsupk V^k ,\tau
 	\right) +  \tilde \gamma\left(\limsupk d^k ,\tau \right)
\end{equation}
where $\tilde \gamma(s,\tau) = \gamma_x^u(\tau) \gamma_u^d(s) + \gamma_x^d(s,\tau)$. The combined restrictions on $V^k$ from Theorems~\ref{thm:plant_LISS} and \ref{thm:algo_LIOS} are $\sup_{k\geq 0} V^k \leq \tilde\eps = \min\{0.25 \eps_x, \bar \veps_2\}$. Thus by the small-gain theorem \cite[Theorem 1]{jiang2004nonlinear}, \cite{karafyllis2007small}, the coupled system is LISS if the condition $\gamma_x^u(\tau) \gamma_u^x(s,\tau) < s$ for all $s\in[0,\tilde \eps]$ holds.
{\hfill $\blacksquare$}

\subsection{Proof of Theorem~\ref{thm:main_theorem}}
\label{ap:main_theorem}
We begin with the following preparatory Lemma. The idea is to prove LISS with respect to a virtual disturbance representing the problematic super-linear part of the gain $\gamma_u^x$.

\begin{lmm} \label{cor:affine_lios}
Given Assumptions~\ref{ass:system_properties} - \ref{ass:algo}, the system \eqref{eq:sys2e} is locally input-output practically stable, i.e., for all $W^0 \leq \bar \veps_1$, $\|V\| \leq \bar \veps_2$, and $\|d\| \leq \bar \veps_3$
\begin{equation} \label{eq:LIOS_affine}
	|\Delta u^k| \leq \beta_u(W^k,k) + \kappa(\tau) \|V\| + \gamma_u^d(\|d\|) + a(\tau)
\end{equation}
where $\kappa(\tau) = \gamma_u^x(\bar \veps_2,\tau)/\bar\veps_2$, $a(\tau) = \max_{s\in [0,\bar \veps_2]}~\gamma_u^x(s,\tau) - \kappa s$, and $\gamma_u^x,\gamma_u^d$, $\bar \veps_i,~ i \in \{1,2,3\}$ are defined in Theorem~\ref{thm:algo_LIOS}. Moreover, $\kappa\in \L$, $a(\tau) \geq 0$, and $a(\tau) \to 0$ as $\tau \to \infty$.
\end{lmm}
\begin{proof}
By Theorem~\ref{thm:algo_LIOS}, we have that
\begin{equation} \label{eq:LIOS_1}
	\|\Delta u^k\| \leq \beta_u(W^k,k) + \gamma_u^x(\|V\|,\tau) + \gamma_u^d(\|d\|)
\end{equation}
for sufficiently small $W^0, \|V\|$, and $\|d\|$. We proceed by bounding $\gamma_u^x$. By construction, for any $\tau > 0$, the affine function $\eta(s,\tau) = \kappa(\tau) s + a(\tau)$ satisfies $\eta(s,\tau) \geq \gamma_u^x(s,\tau)$ for all $s\in [0,\bar\veps_2]$. Substituting this bound into \eqref{eq:LIOS_1} yields \eqref{eq:LIOS_affine}.
The remaining claims follow from $\displaystyle \lim_{\tau \to \infty} a(\tau) = \max_{[0,\bar\veps_2]} \lim_{\tau\to\infty} \gamma_u^x(s,\tau) - \gamma_u^x(\bar \veps_2,\tau)s/\bar\veps_2 = 0$ since $\gamma_u^x\in \KL$.
\end{proof}


Let $V^k = V(x^k,u^k,w^k)$ and $W^k = W(z^k,w^k)$. Following the same steps as in the proof of Theorem~\ref{thm:dt_coupled_iss} but using Lemma~\ref{cor:affine_lios} in place of Theorem~\ref{thm:algo_LIOS} we obtain that
\begin{equation*}
	\limsupk V^k \leq \gamma_x^u(\tau) \kappa(\tau) ~\limsupk V^k + \tilde \gamma\left(\limsupk d^k ,\tau \right) + \gamma_x^u(\tau) a(\tau),
\end{equation*}
where $\tilde \gamma(s,\tau) = \gamma_x^u(\tau) \gamma_u^d(s) + \gamma_x^d(s,\tau)$ and thus the discrete-time system \eqref{eq:discrete-time-error-system} is LISS with respect to the disturbance inputs $\|d\|$ and $a(\tau)$ if the small-gain condition
\begin{equation} \label{eq:small_gain_simple}
	\gamma_x^u(\tau) \kappa(\tau) < 1
\end{equation}
holds in a neighbourhood of the origin \cite[Theorem 1]{jiang2004nonlinear}, \cite{karafyllis2007small}. Since both $\gamma_x^u \in \L$ and $\kappa\in \L$ (and thus go to zero monotonically as $\tau \to \infty$) it follows that there exists $\bar \tau \in (0,\infty)$ such that \eqref{eq:small_gain_simple} is satisfied for all $\tau > \bar \tau$. 

To show LISS of the sampled-data system \eqref{eq:sampled-data-system} we will invoke \cite[Theorem 5]{nevsic1999formulas}. To use this theorem we need to show that (i) \eqref{eq:sampled-data-system} is uniformly bounded over $\tau$ \cite[Definition 2]{nevsic1999formulas}, \blue{i.e., its solutions $(x(t), z(t))$ satisfies $|(\delta x(t),e(t))| \leq \zeta_1(|(\delta x(t^k),e(t^k)|) + \zeta_2(\|d\|_\infty)$ for some $\zeta_1,\zeta_2 \in \K$ over every interval $t\in [t^k~~t^{k+1})$} and (ii) that the discrete-time system \eqref{eq:discrete-time-error-system} is LISS. Condition (i) clearly holds, \blue{$x$ and thus $\delta x$ is bounded over every interval due to \eqref{eq:boundedness}, moreover $z$ is piecewise constant (and thus bounded) and $\mc{W}$ is compact and thus $e(t) = z(t) - \bar z(w(t))$ is also bounded.} As proven above, discrete-time LISS with respect to $d$ and $a(\tau)$ holds for $\tau > \bar \tau$ and thus (ii) holds and \eqref{eq:sampled-data-system} is LISS in continuous-time for sufficiently large $\tau$, i.e., there exists $\beta\in \KL$, $\hat\gamma, \gamma\in \K$ such that $\forall t\geq 0$
\begin{equation}
	\left|\begin{bmatrix}
		\delta x(t) \\ e(t)
	\end{bmatrix} \right| \leq \beta\left( \left|\begin{bmatrix}
		\delta x(0) \\ e(0)
	\end{bmatrix} \right|,t\right) + \gamma(\|d\|) + \hat\gamma(a(\tau)).
\end{equation}
We let $b = \hat\gamma \circ a$ and note that, by definition, $\|d\| = \|\dot w\|$ so we can treat $\dot{w}$ as the disturbance input. Finally, $b(\tau) = \hat \gamma(a(\tau)) \to0$ as $\tau\to\infty$ by Lemma~\ref{cor:affine_lios}.
{\hfill $\blacksquare$}

\bibliography{feedback-eq}
\end{document}